\documentclass[10pt,reqno]{amsart}

\usepackage{amssymb}
\usepackage{amsfonts}
\usepackage{amsthm}
\usepackage{aliascnt}
\usepackage{amsmath}
\usepackage{bbm}
\usepackage{xcolor}
\usepackage{enumerate}
\usepackage{mathabx}
\usepackage{lipsum,booktabs}
\usepackage{cite}
\usepackage[normalem]{ulem}
\usepackage{mathtools}
\usepackage{csquotes}
\usepackage[hidelinks]{hyperref}
\usepackage[nameinlink,noabbrev]{cleveref}
\usepackage{bm}

\DeclarePairedDelimiter{\abs}{\lvert}{\rvert}
\DeclareRobustCommand{\rchi}{{\mathpalette\irchi\relax}}
\newcommand{\irchi}[2]{\raisebox{\depth}{$#1\chi$}} 
\numberwithin{equation}{section}
\allowdisplaybreaks
\newtheorem{theorem}{Theorem}[section]
\newaliascnt{proposition}{theorem}
\newtheorem{proposition}[proposition]{Proposition}
\aliascntresetthe{proposition}
\newaliascnt{lemma}{theorem}
\newtheorem{lemma}[lemma]{Lemma}
\aliascntresetthe{lemma}
\newaliascnt{corollary}{theorem}
\newtheorem{corollary}[corollary]{Corollary}
\aliascntresetthe{corollary}
\newaliascnt{conjecture}{theorem}

\aliascntresetthe{conjecture}
\theoremstyle{definition}
\newaliascnt{definition}{theorem}

\aliascntresetthe{definition}
\newaliascnt{notation}{theorem}

\aliascntresetthe{notation}
\theoremstyle{remark}
\newaliascnt{remark}{theorem}
\newtheorem{remark}[remark]{Remark}
\aliascntresetthe{remark}

\newcommand{\R}{\mathbb{R}}

\newcommand{\Rd}{\mathbb{R}^d}
\newcommand{\Sk}{\mathbb{S}^{d-1}}

\DeclareRobustCommand{\rchi}{{\mathpalette\irchi\relax}}

\renewcommand{\hat}{\widehat}

\newcommand{\scriptA}{\mathcal{A}}

\newcommand{\scriptP}{\mathcal{P}}
\newcommand{\scriptQ}{\mathcal{Q}}
\newcommand{\scriptR}{\mathcal{R}}

\newcommand{\dd}{\mathop{}\!\mathrm{d}}

\newcommand{\qtq}[1]{\quad\text{#1}\quad}

\DeclareMathOperator*{\supp}{supp}

\begin{document}
\title[Besicovitch Compression and Quantitative Endpoint Restriction]{Besicovitch Compression and a Quantitative Failure of Endpoint Fourier Restriction}
\author{Soumyajit Acharyya}

\address{Department of Mathematics, Indian Institute of Technology Bombay, Mumbai 400076}
\email{24d0785@iitb.ac.in}


\begin{abstract}
A classical theorem of Beckner, Carbery, Semmes, and Soria states that the Fourier restriction operator for the sphere does not satisfy the restricted weak-type estimate at the conjectured endpoint. Their proof uses the Besicovitch compression phenomenon and proceeds by contradiction, as in Fefferman's disproof of the ball multiplier conjecture. As a result, it reveals little about how badly the estimate fails. We develop a general framework that converts Besicovitch compression into a quantitative failure at an endpoint. Given a family of sets that compresses, our framework produces an explicit set that violates the bound, with a rate governed by the amount of compression. The rate is sharp within the admissible family of sets. We apply this to the sphere and the paraboloid and construct a deterministic sequence of sets along which the restricted weak-type inequality degrades at an explicitly computed rate. We also show that our construction does not provide blowup for the moment curve in dimensions three and higher, where the (much stronger) restricted strong-type bound is known to hold at the endpoint.
\end{abstract}

\date{\today}

\maketitle


\section{Introduction}

Besicovitch sets of arbitrarily small volume are frequently used in the literature to prove, via the method of contradiction, that certain classical bounds in harmonic analysis cannot be valid. Fefferman's disproof of the ball multiplier conjecture~\cite{Fefferman1971} is perhaps the most well known example of this. This same compression effect also lies at the heart of a classical result of Beckner, Carbery, Semmes, and Soria~\cite{BCSS}, who used an analogous method of contradiction to show that Stein's spherical restriction operator does not obey the restricted weak-type bound at the conjectured endpoint. As the proofs proceed by contradiction, they shed little insight into the extent to which the inequality breaks down. The aim of this article is to present a deterministic process (in \Cref{subsec: selection}) that produces such a counterexample for both the sphere and the paraboloid. This enables us to provide a quantitative rate of failure for the endpoint estimate by determining how it deteriorates as the construction is iterated. Moreover, in \Cref{sec : moment curve}, we explain why our construction fails to provide blowup at the endpoint for the moment curve in dimensions three and higher.

For $f \in L^1(\Rd)$, consider the Fourier transform
$$
\hat{f}(\xi) := \int_{\Rd} f(x) e^{- 2 \pi i x \cdot \xi} \, \dd x, \qquad \xi \in \Rd.
$$
One is interested in Lebesgue space estimates for the Fourier restriction operator for the sphere, identifying pairs $(p, q)$ with $1 \leq p, q \leq \infty$ satisfying
$$
\|\hat f\|_{L^q(\mathbb{S}^{d - 1})} \leq C_{p, q, d} \|f\|_{L^p(\Rd)}
$$
uniformly in $f$ in a dense class, say smooth $f$ with compact support, where on the above LHS we take the $d - 1$ dimensional surface measure on the sphere. This is partly because of its now well-established application to regularity questions for dispersive equations, analytic number theory, such as problems concerning solutions to the Vinogradov system, and geometric combinatorics. Standard scaling/method of stationary phase gives the necessary condition
$$
\frac{d-1}{q} \ge \frac{d+1}{p'} \qquad \text{and} \qquad p < \frac{2d}{d+1}.
$$
Thus, one is led to ask if a weaker Lebesgue space estimate holds at the endpoint $p = q = \frac{2d}{d+1}$; for then one could apply interpolation to get estimates at intermediate points, thereby proving Stein's Fourier restriction conjecture. Unfortunately, as it turns out, and this is a well-known classical result of Beckner, Carbery, Semmes, and Soria~\cite{BCSS} that even the much weaker restricted weak-type bound does not hold at this point, i.e.
$$
\sup_{0 < |E| < \infty} \frac {\|\hat{\rchi_E}\vert_{\mathbb{S}^{d - 1}}\|_{L^{\frac{2d}{d + 1}, \infty}(\mathbb{S}^{d - 1})}}{|E|^{\frac {d+1}{2d}}} = \infty.
$$ 
Here, the supremum is taken over all Lebesgue measurable $E \subset \Rd$ with nonzero finite measure. The proof in~\cite{BCSS}, however, does not quantify the failure as it proceeds through a method of contradiction.

We now proceed to state our result precisely. To do so, we present a general abstract framework for a constructive method, from which the quantitative failure result follows for both the sphere and the paraboloid. We begin by introducing the datum that serves as the input to the deterministic procedure, together with the three quantities associated with it.

\subsection{A datum and its associated set}\label{sec: comptowit}
For $1<p<2$ denote
\begin{equation*}
    \theta_p := \frac{1}{p}- \frac{1}{2}.
\end{equation*}
Let $\Sigma \subset \Rd$ be a Borel set equipped with a finite Borel measure $\mu$. Consider an ordered datum
$\scriptP := \left\{\left(Q_j, \xi_j, \Theta_j\right)\right\}_{j = 1}^N,$ where each $Q_j \subset \Rd$ is bounded and measurable, each $\Theta_j \subset \Sigma$ is $\mu$-measurable, and the point $\xi_j \in \Theta_j$. Suppose that $\abs*{Q_j} = q>0$, for $1 \le j \le N,$ and set 
\[
U : = \bigcup_{j = 1}^N Q_j, \qquad \Omega : = \bigcup_{j = 1}^N \Theta_j, \qquad a := \mu \left(\Omega \right), \qtq{and} S:= \sum_{j = 1}^N \abs*{Q_j} = Nq.
\]
Define the three quantities associated with $\scriptP$ by
\[
\begin{gathered}
     \rho(\scriptP) := \frac{|U|}{S}, \qquad
    b_p(\scriptP) := \frac{ qa^{1/p}}{S^{1/p}},
    \\[4pt]
 \qtq{and}  \kappa(\scriptP) : = \min_{1 \le j \le N} \inf_{\xi \in \Theta_j} \frac{1}{q}  \abs*{\int_{Q_j} e^{-2 \pi i x \cdot \left( \xi - \xi_j\right)}\,\dd x}.
\end{gathered}
\]

\begin{theorem}\label{thm: abstractthm}
    Assume that $a >0$ and $\kappa(\scriptP) >0$. Then there exists a bounded measurable set $E(\scriptP) \subset U$ of positive finite volume such that
     \begin{equation}\label{E : lower estimate}
     \frac{\| \hat{\rchi_{E (\scriptP)}}\|_{L^{p, \infty}(\Sigma, \mu)}}{\abs*{E (\scriptP)}^{1/p}} \ge c_p \, \kappa(\scriptP) \, b_p(\scriptP) \, \rho(\scriptP)^{-\theta_p}.
     \end{equation}
       Here $c_p$ is an absolute constant that depends only on $p$.
\end{theorem}

For the proof, we introduce a set-selection process $\scriptP \mapsto E(\scriptP)$ in \Cref{subsec: selection} and provide a deterministic construction. For every ordered datum $\scriptP$, this produces the set $E(\scriptP)$, although no closed-form description is given. We call $E(\scriptP)$ the prescribed set associated with $\scriptP$.

We now state a few applications of the above result. For this, we need to set up a few notations to be used throughout this article.

\subsection{Notation} For $\Sigma \subset \Rd$, a Borel set equipped with a measure $\mu$, the associated Fourier restriction operator is 
\[
\scriptR_{\Sigma}f := \hat{f}\big|_{\Sigma}.
\]
For a measurable function $F$ on a measure space $(X, \nu)$ and $1 \le r < \infty$, we use
\[
\begin{aligned}
    \|F\|_{L^{r, \infty}(X, \nu)} &:= \sup_{\lambda >0}\lambda \nu(\{|F| > \lambda\})^{1/r},\\
     \|F\|_{L^{r, 1}(X, \nu)} &:= \int_0^\infty \nu(\{|F| > \lambda\})^{1/r}\, \dd \lambda.
\end{aligned}
\]
When the underlying measure space is clear, it is omitted from the notation.

For $1 \le r,s < \infty$, we say that $\scriptR_\Sigma$ is of restricted weak type $(r,s)$ with respect to $\mu$ if there exists $C < \infty$ such that
\[
\|\scriptR_{\Sigma}\rchi_E\|_{L^{s, \infty}(\Sigma, \mu)} \le C|E|^{1/r}
\]
for every measurable $E \subset \Rd$ of finite measure.

For a family $\scriptQ = \left\{Q_j\right\}_{j = 1}^N$ of measurable sets, each with a positive finite volume, we define its compression ratio by 
\begin{equation*}
\text{Comp}\left(\scriptQ\right) := \frac{|\bigcup_{j=1}^{N} Q_j|}{\sum_{j=1}^{N} \abs*{Q_j}}.
\end{equation*}

We use $c$ and $C$ to denote, respectively, small and large positive implicit constants whose exact values may vary from line to line but may depend only on the dimension $d$ and the Lebesgue exponent $p$. Any dependence of implicit constants on further parameters will be indicated by additional subscripts. We write $A \lesssim B$ to mean $A \leq C B$, and $A \sim B$ to mean that both $A \lesssim B$ and $B \lesssim A$ hold.

Beckner, Carbery, Semmes, and Soria consider small spherical caps with bounded overlap and their corresponding dual parallelepipeds. C\'ordoba's sprouting construction \cite{Cordoba} allows these parallelepipeds to be translated so that the volume of their union is small relative to the sum of their volumes. They used this compression to prove the failure of the restriction operator for the sphere to be of restricted weak type at the endpoint. The same compression mechanism also forces the restricted weak-type failure of the restriction operator for the paraboloid at the endpoint. For the moment curve in two dimensions, Bak, Oberlin, and Seeger\cite{BOS}, using a result of Keich\cite{keich}, quantified the failure of the restriction operator to be of restricted weak type at the endpoint. However, in dimensions $d \ge 3$, they proved that the restriction operator is of restricted strong type at the endpoint.

Thus, we also study why the compression mechanism, which is responsible for certain endpoint failures, cannot produce the same type of failure when restricted to the moment curve in higher dimensions. We show that for the moment curve $(t, t^2, \ldots, t^d)$ in dimensions $d \ge 3$, the compression ratio of the dual parallelepipeds associated with any number of $\delta$-separated points on the real line remains bounded below by a positive constant depending only on $d$, regardless of how the parallelepipeds are translated. However, for the moment curve in dimension two, if $N\ge2$ and $t_1, \dots, t_N \in \R$ are  $\delta$-separated, then the compression ratio is at least of order $1/(\log N)$ for all translations of the corresponding parallelepipeds, and this bound is sharp. These results show that the geometry of the moment curve in two dimensions differs from that in dimensions $d \ge 3$. They do not provide a new proof of the result that the restriction operator for the moment curve is of restricted strong type at the endpoint in higher dimensions.

\begin{remark}
The power of $\rho(\scriptP)$ in \eqref{E : lower estimate} comes from the size of the union $U$. Indeed, the square function
\begin{equation}\label{eq: mul}
    M(x) := (\sum_{j=1}^N\rchi_{Q_{j}}(x))^{1/2},
\end{equation}
has support inside $U$ and has mass $\|M\|_2 = \sqrt{S}$. Since $U$ has finite measure, \Cref{lem: fsupportLP1L2} applied to the function $M$ gives
\begin{equation}\label{eq: possibleloss}
    \|M\|_{L^{p,1}} \lesssim \abs*{U}^{1/p-1/2} \|M\|_2 
     = \rho(\scriptP)^{1/p-1/2}S^{1/p}.
\end{equation}
\end{remark}

\subsection{Sharpness of the exponent \texorpdfstring{$\bm{-\theta_p}$}{minus theta p} of \texorpdfstring{$\bm{\rho(\scriptP)}$}{rho(P)}}

It is natural to ask if the exponent of $\rho$ above can be improved. Our next result shows that this is sharp within the class of data covered by \Cref{thm: abstractthm}.

\begin{proposition}\label{prop: sharpness}
    Let $d \ge 1$ and $1<p<2$. For every integer $N \ge 2$, there exists an ordered datum $\scriptP_N := \{(Q_{N,j}, \xi_{N,j}, \Theta_{N,j})\}_{j = 1}^N$, a Borel set $\Sigma_N \subset \Rd$, along with a finite Borel measure $\mu_N$ on it such that
    \[
    \rho\left(\scriptP_N\right)= \frac{1}{N}, \qquad \kappa\left(\scriptP_N\right) = 1, \qquad b_p\left(\scriptP_N\right) = 1.
    \]
 In addition, for $U_N := \bigcup_{j=1}^NQ_{N,j}$, we have
    \[
    \sup_{\substack{E \subset U_N\\ 0< \abs*{E} < \infty}}      \frac{\| \hat{\rchi_{E}}\|_{L^{p, \infty}(\Sigma_N, \mu_N)}}{\abs*{E}^{1/p}} \lesssim\rho\left(\scriptP_N\right)^{-\theta_p}.
    \]
Therefore, the exponent $-\theta_p$ of $\rho(\scriptP)$ in \Cref{thm: abstractthm} is sharp.
\end{proposition}

In fact, we will select $Q_{N,1}= Q_{N,2}= \dots= Q_{N,N} = [0,1]^d$. With $M_N$ is as defined in \eqref{eq: mul}, we have $M_N = \sqrt{N}\rchi_{[0,1]^d}$, and therefore
\[
\|M_N\|_{L^{p,1}(\Rd)} = \sqrt{N} = \rho\left(\scriptP_N\right)^{\theta_p}S_N^{1/p}.
\]
Observe that, with this example, the upper bound in \eqref{eq: possibleloss} is achieved (up to an implicit constant).

\subsection{Applications}

We apply \Cref{thm: abstractthm} to the sphere and the paraboloid to prescribe sequences of sets that make the corresponding restriction operators fail to be of restricted weak type at their diagonal endpoints. In each case, we use the geometry of the surface to construct an ordered datum $\scriptP_m$ for which $\kappa(\scriptP_m)$ and $b_p(\scriptP_m)$ remain bounded below by positive constants independent of $m$, while $\rho(\scriptP_m) \rightarrow0$. At each stage $m$, \Cref{thm: abstractthm} then prescribes $E(\scriptP_m)$ as the desired set.

\subsubsection{Sphere}

Let $\sigma$ denote the surface measure on $\Sk$, $d \ge 2$. The  diagonal endpoint for the restriction to the sphere is $(p_{d}, p_{d})$ where $p_d := 2d/(d+1)$.

For a positive integer $m$, in \Cref{sec: secsph} we use the method in C\'ordoba's sprouting construction from~\cite[Lemma 3, pp.~146--148]{Cordoba} to build an ordered datum $\scriptP_m^{\mathbb{S}} = \{(Q_{m,j}, \xi_{m,j}, \Theta_{m,j})\}_{j=1}^{2^m}$. Here $\Theta_{m,j}$ is the cap of radius $\delta_m = (m2^m)^{-1}$ around $\xi_{m,j} \in \Sk$, and $Q_{m,j}$ is a plate of dimensions $\delta_m^{-2} \times \delta_m^{-1} \times \cdots \times \delta_m^{-1}$ dual to it, positioned by the $j$-th triangle of the $m$-th sprouting stage. Let us denote
$$
E_m^{\mathbb{S}} := E(\scriptP_m^{\mathbb{S}})
$$
for the prescribed set associated with it.

\begin{theorem}\label{thm: sphere}
Let $d \ge 2$. For every sufficiently large $m$, the above set $E_m^{\mathbb{S}}$ satisfies 
  \[
    \frac{\| \hat{\rchi_{E_m^{\mathbb{S}}}}\|_{L^{p_d, \infty}(\Sk, \sigma)}}{\abs*{E_m^{\mathbb{S}}}^{1/p_d}} \ge C_\mathbb{S} \left(\frac{m}{\log m}\right)^{1/(2d)},
    \]
    where $C_\mathbb{S}>0$ is a constant independent of $m$. In particular, restriction to $\Sk$ is not of restricted weak type $\left(p_d, p_d\right)$.
\end{theorem}

The quantities $\rho(\scriptP_m^\mathbb{S})$, $\kappa(\scriptP_m^\mathbb{S})$, and $b_{p_d}(\scriptP_m^\mathbb{S})$ associated with the datum $\scriptP_m^{\mathbb{S}}$ satisfy
\[
 \rho(\scriptP_m^\mathbb{S})\le C_0\frac{\log m}{m}, \qquad  \kappa(\scriptP_m^\mathbb{S}) \ge \frac{1}{2}, \qquad b_{p_d}(\scriptP_m^\mathbb{S}) \ge b_d,
\]
where $C_0 >0$ and $b_d >0$ are constants, independent of $m$. Since $\theta_{p_d}=1/(2d)$, the conclusion follows from \Cref{thm: abstractthm}.

\subsubsection{Paraboloid} \label{sec: parabolaintro}

Let $\Gamma_{d}(\zeta) := \left(\zeta, |\zeta|^2 \right)$ for $\zeta \in \R^{d-1}$, and $d \ge 2$. Let $I \subset \R^{d-1}$ be compact with a nonempty interior. We equip $\Gamma_{d}(I)$ with the pushforward of Lebesgue measure $\mu_{\Gamma_{d}} := (\Gamma_{d})_*(d\zeta)$. The diagonal endpoint for restriction to the paraboloid is $(p_{d}, p_{d})$.

We now employ Keich's family of triangles from~\cite[Theorem 1]{keich} and construct an ordered datum 
\[
\scriptP_m^{\Gamma_d} = \left\{\left(Q_{m,k}, \xi_{m,k}, \Theta_{m,k}\right)\right\}_{k = 0}^{2^m-1}.
\]
Here $\Theta_{m,k} = \Gamma_d(I_{m,k})$, where $I_{m,k}$ is a $d-1$ dimensional cube of length $\delta_m = 2^{-m}$, centered at $\zeta_{m,k} \in I$, and $\xi_{m,k}= \Gamma_d(\zeta_{m,k})$. $Q_{m,k}$ is a parallelepiped of dimensions $\delta_m^{-2} \times \delta_m^{-1} \times \cdots \times \delta_m^{-1}$ dual to $\Theta_{m,k}$, positioned by the $k$-th triangle in Keich's construction. Let $E_m^{\Gamma_d} := E(\scriptP_m^{\Gamma_d})$ denote the bounded set of positive finite volume prescribed by \Cref{thm: abstractthm}.

\begin{theorem}\label{thm: parabola}
    Let $d \ge 2$, and let $I \subset \R^{d-1}$ be compact with nonempty interior. For every sufficiently large $m$, $E_m^{\Gamma_d}$ satisfies
      \[
    \frac{\| \hat{\rchi_{E_m^{\Gamma_d}}} \circ \Gamma_d\|_{L^{p_d, \infty}(I,d\zeta)}}{|E_m^{\Gamma_d}|^{1/p_d}} \ge C_{\Gamma_d,I} m^{1/(2d)},
    \]
    where $C_{\Gamma_d,I}>0$ is a constant independent of $m$. Consequently, the Fourier restriction operator for the paraboloid is not of restricted weak type $\left(p_d, p_d\right)$.
\end{theorem}
The quantities $\rho(\scriptP_m^{\Gamma_d})$, $\kappa(\scriptP_m^{\Gamma_d})$, and $b_{p_d}(\scriptP_m^{\Gamma_d})$ associated with the datum $\scriptP_m^{\Gamma_d}$ satisfy
\[
\rho(\scriptP_m^{\Gamma_d}) < \frac{16}{m}, \qquad \kappa(\scriptP_m^{\Gamma_d}) \ge \frac{1}{2}, \qquad  b_{p_d}(\scriptP_m^{\Gamma_d}) = \beta_{d,I} >0,
\]
where $\beta_{d,I}$ is independent of $m$. Since $\theta_{p_d} = 1/(2d)$, the conclusion follows from \Cref{thm: abstractthm}.

\subsection{Moment-curve obstruction.} \label{sec: momentobs}
Let
\[
\gamma_d(t) := (t, t^2, \dots, t^d), \qquad D:= 1+2+ \dots +d = \frac{d(d+1)}{2},
\]
and let $u_1(t), u_2(t), \dots, u_d(t)$ be the basis such that
\[
\gamma_d^{(k)}(t) \cdot u_r(t) = \delta_{kr}, \qquad 1 \le r, k \le d.
\]
For $t \in \R, \delta >0$ and $c \in \Rd$, define
\[
Q_{\delta,t}(c) := c + \{\sum_{r = 1}^d b_{r}u_{r}(t) : \abs*{b_{r}} \le \delta^{-r}\ \text{for}\ 1 \le r \le d \}.
\]
For the moment curve, for each $t \in \R$, $Q_{\delta,t}(c)$ models a dual parallelepiped corresponding to an interval of length $\delta$ containing $t$. For any $B>0$, the theorem below also holds if $Q_{\delta,t}(c)$ is replaced by $c+ B(Q_{\delta,t}(c)-c)$, but the implicit constants then depend on both $d$ and $B$. We say that a set $T \subset \R$ is $\delta$-separated if $|s-t|\ge \delta$ for any $s \neq t \in T$.

\begin{theorem}\label{thm: momemntobstruct}
    Let $d \ge 3$ and $\delta>0$. Then the following statements hold.

    \begin{enumerate}
        \item[\textup{(i)}]
        If $s,t \in \R$,\ $s \neq t$, and $c_s, c_t \in \Rd$, then
        \begin{equation*} \label{eq: quad-decay}
               \abs*{Q_{\delta,t}(c_t) \cap Q_{\delta,s}(c_s)} \lesssim \frac{\delta^{-\left(D-2\right)}}{|t-s|^2}.
        \end{equation*}

        \item[\textup{(ii)}]
          If $T_\delta \subset \R$ is finite and $\delta$-separated, then, for every choice of centers $\left\{c_t: t \in T_\delta \right\} \subset \Rd$,
\begin{equation*}\label{eq: incompression}
       |\bigcup_{t \in T_\delta}Q_{\delta,t}(c_t)| \gtrsim  \sum_{t \in T_\delta} \abs*{Q_{\delta,t}(c_t)}.
\end{equation*}

    \end{enumerate}
\end{theorem}

 Let $I \subset \R$ be a compact interval and equip $\gamma_d(I)$ with the pushforward measure $\mu_{\gamma_d} := \left(\gamma_d\right)_*\left(dt\right)$. By a single-scale datum at scale $\delta$, we mean an ordered datum
\[
\scriptP = \{(Q_{\delta, t_j}(c_j), \gamma_d(t_j), \gamma_d(I_j)\}_{j = 1}^N,
\]
where each $I_j \subset I$ is an interval,  $t_j \in I_j$, $|I_j| \le A\delta$ for some $A>0$, and $c_j \in \Rd$. Set
\[
r_d :=  \frac{D+1}{D}, \qquad \theta_{r_d} := \frac{1}{r_d}- \frac{1}{2}.
\]
\begin{corollary}\label{cor: momcor}
        Let $d \ge 3$. Then the datum, as above, satisfies
    \begin{equation}\label{eq: cor1}
          b_{r_d}\left(\scriptP\right)^{r_d} \lesssim_A  \rho\left(\scriptP\right).
    \end{equation}
    Consequently,
    \begin{equation}\label{eq: cor2}
           \kappa(\scriptP) b_{r_d}(\scriptP) \rho(\scriptP)^{-\theta_{r_d}} \lesssim_A \rho\left(\scriptP\right)^{1/2} \lesssim_A 1.
    \end{equation}
\end{corollary}

\begin{remark}
The above corollary, in particular, states that no sequence $(\scriptP_m)_{m \ge 1}$ of such data can make the lower bound in \Cref{thm: abstractthm} diverge. If $\rho(\scriptP_m) \rightarrow 0$, then the product on the LHS of \eqref{eq: cor2} also goes to zero.
\end{remark}

To prove \eqref{eq: cor1}, we will choose a maximal $\delta$-separated subcollection of the points $t_1, \dots, t_N$. Maximality gives an upper bound for $ b_{r_d}(\scriptP)$, while \Cref{thm: momemntobstruct}(ii) provides a lower bound for $\rho(\scriptP)$. Combining \eqref{eq: cor1} with the elementary bound $\kappa(\scriptP) \le 1$ yields \eqref{eq: cor2}. The details are given in \Cref{sec : moment curve}.

Our results add to the growing body of literature on sharp quantitative bounds in harmonic analysis, where one asks how badly an estimate fails. In the plane, C\'ordoba \cite{Cordoba1977} established a logarithmic bound for the Kakeya maximal operator, and in \cite[Theorem~1]{keich} Keich showed that this is sharp. Taken together, these results pinpoint the maximum amount of compression possible in a Besicovitch configuration in the plane. Bak, Oberlin and Seeger \cite[(1.2)]{BOS} then used Keich’s theorem to quantify the extent to which the restricted weak-type inequality breaks down at the endpoint for the parabola. Our \Cref{thm: sphere,thm: parabola} provide an analogous failure rate at the conjectured endpoint for the sphere and the paraboloid in all dimensions. A closely related problem is to determine the size of the loss in an inequality that is only known up to an $\varepsilon$ margin. Decoupling is the primary example: the Bourgain–Demeter decoupling theorem~\cite{BourgainDemeter2015} includes an implicit constant of size $O_\varepsilon(\delta^{-\varepsilon})$. Extracting an explicit dependence of the implicit constant is a distinct and difficult issue. Li \cite[Theorem~1.1]{Li2020}, \cite[Theorem~1.1]{Li2021} accomplished this for the parabola, and later, Guth, Maldague and Wang \cite[Theorem~1.1]{GuthMaldagueWang2024} sharpened the loss to a power of $\log(1/\delta)$.

Much of the literature is tied to sharp versions of valid estimates. In this setting, the focus has been on determining the optimal constant, whether there are functions that achieve it and the smoothness of extremizers. This remains difficult even for the Stein-Tomas inequality. Christ and Shao proved the existence of extremizers \cite[Theorem~1.2]{ChristShao2012a} and later established their smoothness in \cite[Theorem~1.1]{ChristShao2012b}; Frank, Lieb, and Sabin \cite[Theorem~1.1]{FrankLiebSabin2016} obtained a dichotomy type result by establishing a lower threshold for the operator norm, ensuring that extremizing sequences are precompact. The sharp constant is known only in a very small number of highly symmetric situations, for instance, for Strichartz inequalities on the paraboloid \cite[Theorems~1.1 and~1.4]{Foschi2007} and on the cone \cite[Theorems~1.5 and~1.7]{Foschi2007}, as well as for the sphere in mixed-norm settings \cite{CarneiroOSousa2019}. For the restriction problem proper, Stovall \cite[Theorem~1.1]{Stovall2020} proved the existence of extremizers for the entire non-endpoint range of $L^p \to L^q$ inequalities for the paraboloid, assuming the restriction conjecture. For curves, the picture is less complete. Drury \cite{Drury1985} resolved the restriction theory for the moment curve, but the optimal constant is still unknown in every dimension. What is known is that extremizers do exist \cite{BiswasStovall2023}, and that sharp estimates are available for monomial curves \cite{BiswasStovall2025}. This was further explored in the finite field case motivated in part by the work of Mockenhaupt and Tao \cite{MockenhauptTao2004}. In that framework, optimal constants have been determined for the parabola, the paraboloid, and certain cones \cite{GonzalezOliveira2024, GonzalezIsmoilov2026}, and also for the moment curve \cite[Theorems~2 and~3]{BCFOST2026} in some special cases.

\begin{remark}
We note that the analog of our results holds for all hypersurfaces with fixed positive Gaussian curvature and for any curve with nonvanishing affine arclength in dimensions three and higher, and the proofs carry over after a local change of coordinates. However, we do not give the details here and leave them to the interested reader.
\end{remark}

\subsection{\textbf{Organization of the paper.}} In \Cref{sec: sec2} we prove \Cref{thm: abstractthm}. In \Cref{sec: sec3} we apply \Cref{thm: abstractthm} to the sphere and the paraboloid, proving \Cref{thm: sphere,thm: parabola}. \Cref{thm: momemntobstruct} and \Cref{cor: momcor} are proved in \Cref{sec : moment curve}. Finally, in \Cref{sec : sec5} we prove \Cref{prop: sharpness}.

\section{Proof of \texorpdfstring{\Cref{thm: abstractthm}}{Theorem 1.1}}\label{sec: sec2}

We now state the selection process $\scriptP \mapsto E(\scriptP)$. The argument has three steps. First, we consider a special class of functions with random signs and show that, upon averaging over the signs, their Fourier transforms have large magnitudes on a set of large measure. Second, we use another averaging argument to define the signs one by one and select a unique function from the class considered in the first step. Finally, we extract a uniquely determined level set from the complex-valued function obtained from the second step.

\subsection{Estimates for averaging over Rademacher signs}
Let $r_j,\ j \ge 1$ denote the Rademacher functions on $[0,1]$, see \cite[Appendix C.$1$]{Grafakos} for a precise definition. We consider the family of functions
\begin{equation}\label{eq: sourcefuncs}
    f_s(x) : = \sum_{j = 1}^N r_j(s) e^{2 \pi i x \cdot \xi_j} \rchi_{Q_j}(x)
\end{equation}
indexed by $s \in [0,1]$. Each $f_s$ has support inside $U$. For $1 \le j \le N$, set
\[
A_j(\xi) := \int_{Q_j} e^{-2 \pi i x \cdot \left( \xi - \xi_j\right)}\,\dd x, \qquad \xi \in \Sigma,
\]
so that $\hat{f_s} = \sum_{j = 1}^N r_j(s) A_j$.
\begin{lemma}\label{lem: khintchine}
    Given $z_1, z_2, \ldots, z_N \in \mathbb{C}$, define $Z(s) := \sum_{j=1}^N r_j(s) z_j$ and let $B := \bigl(\sum_{j=1}^N |z_j|^2\bigr)^{1/2}$. Then there is an absolute constant $\alpha_0$ such that
\[
\left|\left\{\, s \in [0,1] : |Z(s)| \ge \alpha_0 B \,\right\}\right| \ge \alpha_0^2 .
\]
\end{lemma}

\begin{proof}
    If $B = 0$, the assertion is immediate. So we may assume that $B >0$. By Khintchine's inequality, there exists a constant $K_1 >0$ such that 
    $$
        \int_0^1 \abs*{Z(s)}\,\dd s \ge K_1 B.
$$
 Considering the measurable set $E:=\left\{ s \in [0,1]: \abs*{Z(s)} \ge \left(K_1/2\right)B\right\}$, we decompose 
\[
 \int_0^1 \abs*{Z(s)}\,\dd s =  \int_E \abs*{Z(s)}\,\dd s + \int_{E^c} \abs*{Z(s)}\,\dd s.
\]
The second term in RHS is bounded above by $\left(K_1/2\right)B$. On the other hand, by Cauchy–Schwarz and the orthogonality of the Rademacher functions, the first term is bounded above by $B \abs*{E}^{1/2}$. In summary, we conclude that 
$$
K_1 B\le B \abs*{E}^{1/2} + \frac{K_1}{2}B.
$$
Thus, we arrive at the desired conclusion with $\alpha_0 = K_1/2$.
\end{proof}

Now set $\tau := \alpha_0 \kappa(\scriptP) q$ and $Y(s) := \mu ( \{\xi \in \Omega: |\hat{f_s}(\xi)| \ge \tau \})$. For each $\zeta \in \Omega$, we choose a point $\xi \in \Theta_{j(\zeta)}$. We recall that by construction $|A_{j(\zeta)}(\xi)| \ge \kappa(\scriptP) q$, and so
    \[
   ( \sum_{j = 1}^N \abs*{A_j(\xi)}^2 )^{1/2} \ge \kappa(\scriptP) q.
    \]
    Applying \Cref{lem: khintchine} with $z_j := A_j(\xi)$, we get
    \[
    |\{ s \in [0,1]:|\hat{f_s}(\xi)| \ge \tau \}|\ge \alpha_0^2.
    \]
This, in turn, by Tonelli's theorem, gives us
    \begin{equation}\label{eq: outputbound}
           \int_0^1 Y(s)\,\dd s = \int_{\Omega} |\{ s \in [0,1]:|\hat{f_s}(\xi)| \ge \tau \}|\,\dd \mu(\xi) \ge \alpha_0^2 a.
    \end{equation}
We also need to control the average $L^{p,1}$-norm of the family $\{f_s\}$. We obtain this from the following estimate.

\begin{lemma}\label{lem: avgsource}
Let $1 < p < \infty$ and $g_1, \dots , g_N$ be measurable functions on $\Rd$. For $s \in [0,1]$, set
    \[
    F_s(x) := \sum_{j = 1}^Nr_j(s)g_j(x), \qtq{and} G(x):= (\sum_{j = 1}^N \abs*{g_j(x)}^2)^{1/2}.
    \]
    Then the following square root cancellation holds:
    \begin{equation*}
        \int_0^1\|F_s\|_{L^{p, 1}\left(\Rd\right)}\,\dd s \lesssim \|G\|_{L^{p, 1}\left(\Rd\right)}.
    \end{equation*}
\end{lemma}

\begin{proof}
Applying the estimate in \cite[Appendix~C.3, pp.~586--587]{Grafakos} to the real and imaginary parts, for each level set at level $\lambda > 0$ we obtain 
\begin{equation}\label{eq: complextail}
        |\{s: |\sum_{j = 1}^N r_j(s) z_j| > \lambda\}| \le 4 \exp (- \frac{\lambda^2}{4 \sum_{j=1}^N \abs*{z_j}^2})
\end{equation}
for any finitely many complex numbers $z_1, \ldots, z_N$. For a measurable function $h$ on $\Rd$, write $ d_h(\lambda) := |\{x \in \Rd: \abs*{h(x)} > \lambda\}|$.

Jensen's inequality applied to the concave function $u \mapsto u^{1/p}$ yields
\begin{equation}\label{eq: sourceconcave}
         \int_0^1\|F_s\|_{L^{p, 1}}\,\dd s = \int_0^\infty \int_0^1  d_{F_s}(\lambda)^{1/p}\, \dd s \dd \lambda \le \int_0^\infty (\int_0^1  d_{F_s}(\lambda)\, \dd s)^{1/p}\, \dd \lambda.
    \end{equation}
We now apply \eqref{eq: complextail} with $z_j = g_j(x)$ to get
\[
\int_0^1d_{F_s}(\lambda)\, \dd s \le 4\int_{\left\{G>0\right\}}\exp(- \frac{\lambda^2}{4 G(x)^2})\,\dd x.
\]
Let us denote the above RHS as $H(\lambda)$. For fixed $\lambda >0$, we decompose the set
\[
\{G>0 \} = \{G>\lambda \} \bigcup_{k=0}^\infty \{2^{-k-1}\lambda < G \le 2^{-k}\lambda \}.
\]
  On the set $\{2^{-k-1}\lambda < G \le 2^{-k}\lambda \}$, we have $\lambda^2/4G^2 \ge 4^{k-1}$; hence
  \[
  H(\lambda) \le 4d_G(\lambda) + 4 \sum_{k = 0}^\infty e^{{-4}^{k-1}}d_G(2^{-k-1}\lambda).
  \]
Changing variables $u = 2^{-k-1}\lambda$, we obtain
  \[
  \begin{aligned}
       \int_0^\infty H(\lambda)^{1/p}\,\dd \lambda &\lesssim\int_0^\infty d_G(\lambda)^{1/p}\,\dd \lambda\ +  \sum_{k = 0}^\infty 2^{k+1}e^{-4^{(k-1)}/p} \int_0^\infty d_G(u)^{1/p}\,\dd u \lesssim\|G\|_{L^{p,1}}.
  \end{aligned}
 \]
 Combining this with \eqref{eq: sourceconcave} finishes the proof.
\end{proof}

 Applying the above \Cref{lem: avgsource} for $ g_j(x) := e^{2 \pi i x \cdot \xi_j} \rchi_{Q_j}(x)$, we deduce that
\begin{equation}\label{eq: avsoimp}
        \int_0^1\|f_s\|_{L^{p, 1}\left(\Rd\right)}\,\dd s \lesssim \|M\|_{L^{p, 1}\left(\Rd\right)},
    \end{equation}
    where $M$ is as defined in \Cref{sec: comptowit}.

We will make use of the following basic estimate for functions whose support has finite volume.

\begin{lemma}\label{lem: fsupportLP1L2}
    Let $1<p<2$. Let $h \in L^2(\Rd)$ be supported in a set $V \subset \Rd$ of finite measure, then it holds that
    \[
    \|h\|_{L^{p,1}(\Rd)} \lesssim |V|^{1/p-1/2}\|h\|_{L^{2}(\Rd)}.
    \]
\end{lemma}

\begin{proof}
For $h = 0$, the conclusion is immediate. We may assume that $h \neq 0$, and so $|V|>0$. Now, for each $\lambda > 0$, clearly
    \[
    d_h(\lambda) \le \min (\abs*{V}, \lambda^{- 2} \|h\|_{L^{2}(\Rd)}^2).
    \]
Write $\lambda_0 = \|h\|_{L^{2}(\Rd)}\abs*{V}^{-1/2}$. Since $p<2$, we get 
\[
 \begin{aligned}
            \|h\|_{L^{p,1}(\Rd)} &\le \int_0^{\lambda_0} \abs*{V}^{1/p}\, \dd \lambda\ +\ \|h\|_{L^{2}(\Rd)}^{2/p} \int_{\lambda_0}^{\infty}\lambda^{-2/p}\, \dd \lambda \lesssim \abs*{V}^{1/p-1/2}\|h\|_{L^{2}(\Rd)},
    \end{aligned}
\]
completing the proof.
\end{proof}

Combining \eqref{eq: avsoimp} with \eqref{eq: possibleloss}, we obtain

    \begin{equation}\label{eq: sourcebound}
        \int_0^1\|f_s\|_{L^{p, 1}\left(\Rd\right)}\,\dd s \lesssim \rho(\scriptP)^{\theta_p}S^{1/p}.
    \end{equation}
    We fix an implicit constant, which we denote by $C_p^{\mathrm{av}}$, so that the above estimate \eqref{eq: sourcebound} holds.
   \subsection{A deterministic selection of signs}
We now use the two estimates \eqref{eq: outputbound} and \eqref{eq: sourcebound} to choose a sign vector $\varepsilon^* = \left( \varepsilon_1^*, \varepsilon_2^*, \dots , \varepsilon_N^* \right) \in \left\{1, -1 \right\}^N$ and uniquely determine one function from the family \eqref{eq: sourcefuncs}.  

For each  $\varepsilon = \left( \varepsilon_1, \varepsilon_2, \dots , \varepsilon_N \right) \in \left\{1, -1 \right\}^N$, define
\[
   f_\varepsilon(x) : = \sum_{j = 1}^N \varepsilon_j e^{2 \pi i x \cdot \xi_j} \rchi_{Q_j}(x),
\]
and put
\[
   X(\varepsilon):=\|f_\varepsilon\|_{L^{p, 1}\left(\Rd\right)}, \qquad Y(\varepsilon) := \mu ( \{\xi \in \Omega: |\hat{f_\varepsilon}(\xi)| \ge \tau \}).
\]
For each $\varepsilon \in \left\{1, -1 \right\}^N$, let
\[
D_\varepsilon := \{ s \in [0,1]: r_j(s) = \varepsilon_j, \ 1 \le j \le N\}.
\]
The sets $D_\varepsilon$ are pairwise disjoint, each has measure $2^{-N}$, and they cover $[0,1]$ up to finitely many points. Moreover, $f_s = f_\varepsilon$ on $D_\varepsilon$. Consequently, \eqref{eq: outputbound} and \eqref{eq: sourcebound} give

\begin{equation}\label{eq: discavg}
    2^{-N}\sum_{\varepsilon \in \left\{1,-1\right\}^N }Y(\varepsilon) \ge \alpha_0^2 a, \qquad 2^{-N}\sum_{\varepsilon \in \left\{1, -1 \right\}^N } X(\varepsilon) \leq C_p^{\mathrm{av}} \rho(\scriptP)^{\theta_p}S^{1/p}.
\end{equation}
Set
\begin{equation*}
    \Lambda_p := C_p^{\mathrm{av}} \rho(\scriptP)^{\theta_p}S^{1/p}, \qquad c_*:= \frac{\alpha_0^2 a}{2 \Lambda_p}, \qquad F(\varepsilon) := Y(\varepsilon) - c_*X(\varepsilon).
\end{equation*}
It follows from \eqref{eq: discavg} that
\[
 2^{-N}\sum_{\varepsilon \in \left\{1, -1 \right\}^N } F(\varepsilon) \ge \alpha_0^2 a - c_* \Lambda_p = \frac{\alpha_0^2 a}{2} >0,
\]
since $\alpha_0^2, a >0$.

For $0 \le k \le N$ and $\eta = \left( \eta_1, \eta_2, \dots ,\eta_k \right) \in \left\{1, -1 \right\}^k$, define 
\[
\scriptA(\eta) : = 2^{-(N-k)}\sum_{\omega \in \left\{1, -1 \right\}^{N-k}} F\left(\eta_1, \dots, \eta_k, \omega_1, \dots ,\omega_{N-k} \right).
\]
For the empty vector $\emptyset$, this gives
\[
\scriptA(\emptyset) = 2^{-N}\sum_{\varepsilon \in \left\{1, -1 \right\}^N } F(\varepsilon) \ge \frac{\alpha_0^2 a}{2} >0.
\]
Suppose that $\varepsilon_1^*, \varepsilon_2^*, \dots , \varepsilon_k^*$ have been chosen for some $0 \le k <N$. Define 
\[
\scriptA_+:= \scriptA(\varepsilon_1^*, \varepsilon_2^*, \dots , \varepsilon_k^*,1), \qquad \scriptA_-:= \scriptA(\varepsilon_1^*, \varepsilon_2^*, \dots , \varepsilon_k^*,-1).
\]
Then set the next sign
\begin{equation}\label{eq: sign-recursion}
    \varepsilon_{k+1}^* := 
    \begin{cases}
        +1, &\scriptA_+ \ge \scriptA_-,\\
        -1, &\scriptA_- > \scriptA_+.
    \end{cases}
\end{equation}
Starting from the empty vector and applying this rule repeatedly, we determine all $N$ signs uniquely.

For every $\eta = \left( \eta_1, \eta_2, \dots ,\eta_k \right) \in \left\{1, -1 \right\}^k$, $0 \le k < N$,
\[
\scriptA(\eta) = \frac{1}{2}\scriptA(\eta,1)\ + \ \frac{1}{2}\scriptA(\eta,-1).
\]
Hence $\scriptA\left(\varepsilon_1^*, \dots , \varepsilon_{k+1}^*\right) \ge \scriptA\left(\varepsilon_1^*, \dots , \varepsilon_{k}^*\right)$ for $0 \le k < N$. Consequently,

\begin{equation}\label{eq: oularin}
  F\left(\varepsilon^*\right) = \scriptA\left(\varepsilon_1^*, \varepsilon_2^*, \dots , \varepsilon_N^*\right) \ge \scriptA \left(\emptyset\right) \ge \frac{\alpha_0^2 a}{2} > 0,  
\end{equation}
where $\varepsilon^*:= \left(\varepsilon_1^*, \varepsilon_2^*, \dots , \varepsilon_N^*\right)$.

Define
\begin{equation}\label{eq: main function}
    f_*(x) := f_{\varepsilon^*}(x) = \sum_{j = 1}^N \varepsilon_j^*e^{2 \pi i x \cdot \xi_j} \rchi_{Q_j}(x).
\end{equation}
By the definition of $F$ and \eqref{eq: oularin}, $Y(\varepsilon^*) > c_*X(\varepsilon^*)$. In particular, $Y(\varepsilon^*) > 0$, which implies $\hat{f_*} \not \equiv 0$, and hence $X(\varepsilon^*) = \|f_*\|_{L^{p,1}(\Rd)} >0$. 

Using the level $\tau/2$, we obtain
\begin{equation}\label{eq: fundamental1}
  \|\hat{f_*}\|_{L^{1, \infty}(\Omega, \mu)} \ge \frac{\tau}{2}Y\left(\varepsilon^*\right) > \frac{\tau c_*}{2}\|f_*\|_{L^{p, 1}(\Rd)}.
\end{equation}
For every measurable $H$ on $\Omega$, 
\begin{equation}\label{eq: lorentsctmnt}
  \|H\|_{L^{1, \infty}(\Omega, \mu)} \le a^{1/p'}\|H\|_{L^{p, \infty}(\Omega, \mu)}.  
\end{equation}
Indeed, if $ E_\lambda = \{\xi \in \Omega: |H(\xi)| > \lambda \}$, then
\[
\lambda \mu \left(E_\lambda\right) = \lambda \mu \left(E_\lambda\right)^{1/p}\mu \left(E_\lambda\right)^{1/p'} \le a^{1/p'}\|H\|_{L^{p, \infty}(\Omega, \mu)},
\]
and we take the supremum over $\lambda > 0$.

By the definition of $b_p(\scriptP)$, \eqref{eq: fundamental1} and \eqref{eq: lorentsctmnt} give the desired estimate:
\begin{equation}\label{eq: funmain}
    \|\hat{f_*}\|_{L^{p, \infty}(\Sigma, \mu)} \ge c_p\kappa(\scriptP) b_p(\scriptP) \rho(\scriptP)^{-\theta_p} \|{f_*}\|_{L^{p, 1}(\Rd)}.
\end{equation}
Here $c_p>0$ is a constant depending only on $p$.

\subsection{A deterministic construction of sets}\label{subsec: selection}

Fix an enumeration $\mathbb{Q}_+ = \{q_1, q_2, \dots\}$ of the positive rationals. We first briefly discuss the key steps of the final selection procedure, producing the desired set. The full details are recorded in the following proposition in a general setting.

First, write $f_* = h_1 -h_2 +i(h_3 -h_4)$ , where  $ h_1 := (\text{Re} f_*)_+$, $h_2 := (\text{Re} f_*)_-$, $h_3 := (\text{Im} f_*)_+$ and $h_4:=(\text{Im} f_*)_-$. For $1 \le m \le 4$ and $\lambda>0$, we set $  E_{m, \lambda} := \{ h_m > \lambda\}$ and show that there exist $m_0 \in \{ 1,2,3,4\}$ and $\lambda_0 >0$ such that $|E_{m_0, \lambda_0}|>0$ and 
    \[
    \frac{\|\hat{\rchi_{E_{m_0, \lambda_0}}}\|_{L^{p, \infty}\left(\Sigma, \mu\right)}}{\abs*{E_{m_0, \lambda_0}}^{1/p}} > \frac{1}{8p'}\frac{  \|\hat{f_*}\|_{L^{p, \infty}\left(\Sigma, \mu\right)}}{\|f_*\|_{L^{p, 1}(\Rd)}}.
    \]
 Next, we pass to a rational level by proving the existence of a $q_{\ell_0} \in \mathbb{Q}_+$ such that $|E_{m_0, q_{\ell_0}}|>0$ and
\begin{equation*} 
    \frac{\|\hat{\rchi_{E_{m_0, q_{\ell_0}}}}\|_{L^{p, \infty}\left(\Sigma, \mu\right)}}{\abs*{E_{m_0, q_{\ell_0}}}^{1/p}} \ge \frac{1}{8p'}\frac{\|\hat{f_*}\|_{L^{p, \infty}\left(\Sigma, \mu\right)}}{\|f_*\|_{L^{p, 1}(\Rd)}}.
\end{equation*}
Finally, we order the pairs in $\{ 1,2,3,4\} \times \mathbb{N}$ as
    \[
    (1,1),\ (2,1), \ (3,1),\ (4,1),\ (1,2),\ (2,2),\ (3,2),\ (4,2),\dots,
    \]
select the first pair $(m_*, \ell_*)$ in this ordering for which the above estimate holds, and take 
   $E(f_*) : = E_{m_*, q_{\ell_*}}$ as the uniquely determined set. 
\begin{proposition}\label{prop: exprop}
Fix $1<p<\infty$. Let $\Sigma\subset \Rd$ be a Borel set endowed with a finite Borel measure $\mu$, and let $U\subset \Rd$ be bounded and measurable. Assume $f$ is bounded and measurable, has support contained in $U$, and satisfies $\|\hat f\|_{L^{p,\infty}(\Sigma,\mu)}>0$. Then,  the set $E(f)$ obtained from the selection procedure described above is measurable, contained in $U$, has finite volume, and satisfies 
\[
\frac{\|\hat{\rchi_{E(f)}}\|_{L^{p,\infty}(\Sigma,\mu)}}{\abs*{E(f)}^{1/p}}
\ge \frac{1}{8p'}\,\frac{\|\hat f\|_{L^{p,\infty}(\Sigma,\mu)}}{\|f\|_{L^{p,1}(\Rd)}}.
\]
\end{proposition}
 
\begin{proof}
Write $f = g_1 -g_2 +i(g_3 -g_4)$ ,  where  $ g_1 := (\text{Re} f)_+$, $g_2 := (\text{Re} f)_-$, $g_3 := (\text{Im} f)_+$, $g_4:=(\text{Im} f)_-$. Each $g_m$ is nonnegative, supported in $U$, and bounded above by $\abs*{f}$. For $1 \le m \le 4$ and $\lambda>0$, set
    \[
    E_{m, \lambda} := \left\{x \in \Rd: g_m(x) > \lambda\right\}.
    \]
Then $E_{m, \lambda} \subset U$, and hence $\abs*{E_{m, \lambda}} < \infty$.

Next, set
    \[
    B := \sup_{\substack{1 \le m \le 4,\ \lambda >0\\\abs*{E_{m, \lambda}} >0}} \frac{\|\hat{\rchi_{E_{m, \lambda}}}\|_{L^{p, \infty}\left(\Sigma, \mu\right)}}{\abs*{E_{m, \lambda}}^{1/p}}.
    \]
    Since $f$ is nonzero, at least one $g_m$ is positive on a set of positive measure. Consequently, there exist $m \in \left\{ 1,2,3,4\right\}$ and $k \in \mathbb{N}$ such that $|E_{m, 1/k}| >0$. Hence, the family over which the supremum is taken is nonempty. Moreover, $B < \infty$. Indeed, for every measurable $F \subset U$, $|\hat{\rchi_F}(\xi)| \le \abs*{F}$ for $\xi \in \Sigma$, and therefore
    \[
    \frac{\|\hat{\rchi_F}\|_{L^{p, \infty}\left(\Sigma, \mu\right)}}{\abs*{F}^{1/p}} \le \mu \left(\Sigma\right)^{1/p}\abs*{F}^{1/p'} \le \mu \left(\Sigma\right)^{1/p}\abs*{U}^{1/p'}.
    \]
    Let $c_1 = 1$, $c_2 = -1$, $c_3 = i$, $c_4 = -i$. Since $g_m \ge 0$,  for every $x$, 
    \[
    g_m(x) = \int_0^\infty \rchi_{E_{m, \lambda}}(x)\, \dd \lambda.
    \]
   Since each $g_m$ is integrable, Fubini's theorem gives
    \[
    \hat{f} (\xi) = \sum_{m=1}^4c_m \int_0^\infty \hat{\rchi_{E_{m, \lambda}}}(\xi)\, \dd \lambda.
    \]
    Using an equivalent norm on $L^{p, \infty}$ \cite[Exercise 1.1.12]{Grafakos} and Minkowski's integral inequality, together with the definition of $B$, and the monotonicity of Lorentz norms, we obtain
    \[
    \|\hat{f}\|_{L^{p, \infty}\left(\Sigma, \mu\right)} \le p' \sum_{m=1}^4 \int_0^\infty \|\hat{\rchi_{E_{m, \lambda}}}\|_{L^{p, \infty}\left(\Sigma, \mu\right)}\, \dd \lambda \le 4p'B  \|f\|_{L^{p, 1}(\Rd)}.
    \]
    Therefore,
    \[
        B \ge \frac{1}{4p'} \frac{  \|\hat{f}\|_{L^{p, \infty}\left(\Sigma, \mu\right)}}{\|f\|_{L^{p, 1}(\Rd)}}.
    \]
    By the definition of $B$, there exist $m_0 \in \left\{ 1,2,3,4\right\}$ and $\lambda_0 >0$ such that $\abs*{E_{m_0, \lambda_0}}>0$ and 
    \[
    \frac{\|\hat{\rchi_{E_{m_0, \lambda_0}}}\|_{L^{p, \infty}\left(\Sigma, \mu\right)}}{\abs*{E_{m_0, \lambda_0}}^{1/p}} > \frac{1}{8p'}\frac{  \|\hat{f}\|_{L^{p, \infty}\left(\Sigma, \mu\right)}}{\|f\|_{L^{p, 1}(\Rd)}}.
    \]
    Choose a sequence $(\lambda_r)_{r \ge 1} \subset \mathbb{Q}_+$ such that $\lambda_r > \lambda_0$ and $\lambda_r \downarrow \lambda_0$. Then $E_{m_0, \lambda_r} \uparrow E_{m_0, \lambda_0}$. Continuity from below gives $|E_{m_0, \lambda_r}| \longrightarrow |E_{m_0, \lambda_0}|$ and, consequently, $|E_{m_0, \lambda_0} \setminus E_{m_0, \lambda_r}| \longrightarrow 0$. Moreover, for every $\xi \in \Sigma$,
    \[
     |\hat{\rchi_{E_{m_0, \lambda_r}}}(\xi) - \hat{\rchi_{E_{m_0, \lambda_0}}}(\xi)| \le |E_{m_0, \lambda_0} \setminus E_{m_0, \lambda_r}|.
    \]
    If $H$ and $K$ are measurable on $\Sigma$ and $|H-K|\le \eta$ on $\Sigma$ pointwise, then
 \begin{equation}\label{eq: normcontinuity}
     |\|H\|_{L^{p, \infty}\left(\Sigma, \mu\right)}- \|K\|_{L^{p, \infty}\left(\Sigma, \mu\right)}| \le \eta \mu\left(\Sigma\right)^{1/p}.
 \end{equation}
    Indeed, if $\alpha > \eta$, then $\{|H|>\alpha\} \subset \{|K|>\alpha - \eta\}$, and therefore
    \[
    \begin{aligned}
          \alpha \mu \left\{\abs*{H}>\alpha\right\}^{1/p} &\le  \left(\alpha-\eta\right) \mu \left\{\abs*{K}>\alpha - \eta\right\}^{1/p} \ + \ \eta \mu \left\{\abs*{K}>\alpha - \eta \right\}^{1/p}
          \\
          &\le \|K\|_{L^{p, \infty}\left(\Sigma, \mu\right)} \ + \ \eta \mu\left(\Sigma\right)^{1/p}.
    \end{aligned}
  \]
  For $0<\alpha \le \eta$, the LHS is at most $\eta \mu\left(\Sigma\right)^{1/p}$. Taking the supremum over $\alpha > 0$ gives
  \[
  \|H\|_{L^{p, \infty}\left(\Sigma, \mu\right)} \le \|K\|_{L^{p, \infty}\left(\Sigma, \mu\right)}\ + \ \eta \mu\left(\Sigma\right)^{1/p}.
  \]
  Interchanging the roles of $H$ and $K$ proves \eqref{eq: normcontinuity}.

  It follows that
  \[
    \frac{\|\hat{\rchi_{E_{m_0, \lambda_r}}}\|_{L^{p, \infty}\left(\Sigma, \mu\right)}}{\abs*{E_{m_0, \lambda_r}}^{1/p}}  \longrightarrow\frac{\|\hat{\rchi_{E_{m_0, \lambda_0}}}\|_{L^{p, \infty}\left(\Sigma, \mu\right)}}{\abs*{E_{m_0, \lambda_0}}^{1/p}}.
  \]
  Therefore, there exist $m \in \left\{1, 2, 3, 4\right\}$ and $q_{\ell} \in \mathbb{Q}_+$ such that

\begin{equation} \label{eq: fundamental2}
      \abs*{E_{m, q_{\ell}}}>0, \qquad \frac{\|\hat{\rchi_{E_{m, q_{\ell}}}}\|_{L^{p, \infty}\left(\Sigma, \mu\right)}}{\abs*{E_{m, q_{\ell}}}^{1/p}} \ge \frac{1}{8p'}\frac{\|\hat{f}\|_{L^{p, \infty}\left(\Sigma, \mu\right)}}{\|f\|_{L^{p, 1}(\Rd)}}.
\end{equation}
Order the pairs in $\left\{ 1,2,3,4\right\} \times \mathbb{N}$ as
    \[
    (1,1),\ (2,1), \ (3,1),\ (4,1),\ (1,2),\ (2,2),\ (3,2),\ (4,2),\dots.
    \]
   Let $(m_*, \ell_*)$ be the first pair in this ordering for which \eqref{eq: fundamental2} holds, and define 
   \[
   E(f) : = E_{m_*, q_{\ell_*}}.
   \]
   Thus $E(f)$ is uniquely determined, is contained in $U$, and has positive finite volume. The asserted inequality now follows.
\end{proof}

 \begin{proof}[Proof of \Cref{thm: abstractthm}]

 The sign-selection rule \eqref{eq: sign-recursion} produces the uniquely determined function $f_*$ defined in \eqref{eq: main function}. This function is bounded, measurable, supported in the bounded set $U$, and $\|\hat{f}_*\|_{L^{p, \infty}\left(\Sigma, \mu\right)}>0$. We may therefore apply \Cref{prop: exprop} to $f_*$ and define
 \[
 E\left( \scriptP\right) := E\left(f_*\right).
 \]
Then $ E\left( \scriptP\right) \subset U$ is bounded and measurable, with $0< \abs*{E\left( \scriptP\right)}<\infty$. By \eqref{eq: funmain},
\[
  \frac{\|\hat{\rchi_{E(\scriptP)}}\|_{L^{p, \infty}\left(\Sigma, \mu\right)}}{\abs*{E(\scriptP)}^{1/p}} \ge \frac{1}{8p'}  \frac{\|\hat{f_*}\|_{L^{p, \infty}\left(\Sigma, \mu\right)}}{\|f_*\|_{L^{p,1}(\Rd)}} \ge  c_p\kappa(\scriptP) b_p(\scriptP) \rho(\scriptP)^{-\theta_p}.
\]
This completes the proof.
\end{proof}

\section{Applications}\label{sec: sec3}
\subsection{Sphere}\label{sec: secsph}
We use C\'ordoba's sprouting construction in \cite{Cordoba} starting from the triangle with vertices at $(\frac{-1}{2}, 0)$, $(\frac{1}{2}, 0)$, $(0,1)$. At stage $m$, fix an ordering of the $2^m$ triangles and denote them by $T_{m,1},T_{m,2}, \dots ,T_{m,2^m}$. For each $j$, write
$$
T_{m,j} = \operatorname{conv}\left\{A_{m,j}, B_{m,j}, C_{m,j}\right\},
$$
where the vertices are labeled so that $A_{m,j}B_{m,j}$ is the base, $A_{m,j}C_{m,j}$ is the chosen long side, and $\left(B_{m,j} - A_{m,j}\right) \cdot \left(C_{m,j} - A_{m,j}\right) \ge 0$. For each $1 \le j  \le 2^m$, define 
\[
L_{m,j} : = \abs*{C_{m,j} - A_{m,j}}, \qquad \omega_{m,j} := \frac{C_{m,j} - A_{m,j}}{\abs*{C_{m,j} - A_{m,j}}}.
\]
The sprouting construction gives a constant $C_{\mathrm{spr}}> 0$, independent of $m$, such that for every $m \ge 2$, 
\begin{equation*}
    \abs*{B_{m,j}-A_{m,j}} = 2^{-m}, \qquad \abs*{C_{m,j}-A_{m,j}} \ge  m,
\end{equation*}
and
\begin{equation}\label{eq: triun}
    |\bigcup_{j = 1}^{2^m}T_{m,j}| \le C_{\mathrm{spr}} \log m.
\end{equation}
Moreover, Beckner, Carbery, Semmes, and Soria~\cite[Lemma~3]{BCSS} show that, for sufficiently large $m$, there exists a constant $c_{\mathrm{sep}} > 0$, independent of $m$, such that 
\begin{equation}\label{eq: dirsep}
    \abs*{\omega_{m,j} - \omega_{m,k}} \ge \frac{c_{\mathrm{sep}}}{m2^m}, \qquad 1 \le j \neq k \le 2^m.
\end{equation}
Let $m$ be sufficiently large.  For each $j$, let $\tau_{m,j}$ be the unique unit vector perpendicular to $\omega_{m,j}$ such that
\[
(B_{m,j} - A_{m,j}) \cdot\tau_{m,j} > 0.
\]
Put
\[
a_{m,j} := (B_{m,j} - A_{m,j}) \cdot\omega_{m,j}, \qquad d_{m,j} := (B_{m,j} - A_{m,j}) \cdot\tau_{m,j}.
\]
By the choice of the vertex labeling, $0 \le a_{m,j} \le 2^{-m}$, the construction also gives a constant $c_1 > 0 $, independent of $m$ and $j$, such that $d_{m,j} \ge c_12^{-m}$. Fix such a $c_1$.
Writing  $T_{m,j}$ as
\begin{equation*}\label{eq: trieq}
\begin{aligned}
       T_{m,j} = A_{m,j} + \{ x \omega_{m,j} + y\tau_{m,j}:{}\;&
        0 \le y \le d_{m,j},\\
      & \frac{ a_{m,j}}{ d_{m,j}}y \le x \le L_{m,j} + \frac{ a_{m,j} - L_{m,j}}{ d_{m,j}}y \}
\end{aligned}
\end{equation*}
we set 
$$
Q_{m,j}^1 :=  A_{m,j} + \{x \omega_{m,j} + y\tau_{m,j}: \frac{L_{m,j}}{8} \le x  \le \frac{L_{m,j}}{4} \, \text{and} \, \frac{d_{m,j}}{4} \le y \le \frac{d_{m,j}}{2} \}.
$$
Thus, we have that $ Q_{m,j}^1 \subset  T_{m,j}$.

Let $c: =\min\{c_1/4, 1/8\}$, and  inside $ Q_{m,j}^1$, define the subrectangle 
\begin{equation*}
        \begin{aligned}
       Q_{m,j}^2 =  A_{m,j} + \{  x \omega_{m,j} + y\tau_{m,j}: \frac{L_{m,j}}{8} \le x & \le \frac{L_{m,j}}{8}+cm,\\
       &\frac{d_{m,j}}{4} \le y \le \frac{d_{m,j}}{4} + c2^{-m}\}.
\end{aligned}
\end{equation*}
Each $ Q_{m,j}^2$ has dimensions $cm \times c2^{-m}$.

Next, we consider the following centered version of $Q_{m,j}^2$ given by
\[
Q_{m,j,c}^2:= \{  x \omega_{m,j} + y\tau_{m,j}: |x| \le \frac{cm}{2}, \ |y| \le \frac{c2^{-m}}{2}\}.
\]
Since $Q_{m,j}^2 \subset T_{m,j}$, it follows from \eqref{eq: triun} that
\begin{equation}\label{eq: compresssphere}
    \frac{|\bigcup_{j=1}^{2^m} Q_{m,j}^2|}{\sum_{j=1}^{2^m} |Q_{m,j}^2|} \le C_0 \frac{\log m}{m}, \qtq{where} C_0 : = \frac{C_{\mathrm{spr}}}{c^2}.
\end{equation}
Set $\delta_m : = 1/(m2^m)$ and $r_m : = m2^{2m}$. Write $\omega_{m,j} = (\omega_{m,j}^1, \omega_{m,j}^2)$,\ $\tau_{m,j} = (\tau_{m,j}^1, \tau_{m,j}^2)$, and define
\[
\xi_{m,j} : = ( \omega_{m,j}^1, \omega_{m,j}^2,0, \dots, 0) \in \Sk, \qquad \xi_{m,j} ^\perp: = ( \tau_{m,j}^1, \tau_{m,j}^2,0, \dots, 0) \in \Sk.
\]
Let 
\[
I_m = [ -\frac{c}{2} \delta_m^{-1}, \frac{c}{2} \delta_m^{-1}]^{d-2}.
\]
When $d = 2$, we take the measure of $I_m$ to be $1$. Define
\[
Q_{m,j,c}^3:=r_mQ_{m,j,c}^2 \times I_m \subset \Rd.
\]
If $\lambda_{m,j,0} : = (\lambda_{m,j,0}^1,\lambda_{m,j,0}^2)$ is the center of $Q_{m,j}^2$, define
\[
 \widetilde\lambda_{m,j} : = (r_m\lambda_{m,j,0}^1,r_m\lambda_{m,j,0}^2,0,\dots,0) \in \Rd. \qquad Q_{m,j}^3:=  \widetilde\lambda_{m,j} + Q_{m,j,c}^3.
\]
Take $\eta := 1/(6cd)$ and define
\[
 Q_{m,j} := \eta Q_{m,j}^3, \qtq{and}  \lambda_{m,j} := \eta \widetilde\lambda_{m,j}.
\]
Thus 
\begin{equation}\label{eq: Qmj}
    \begin{aligned}
 Q_{m,j} =  \lambda_{m,j} + \{s_\parallel \xi_{m,j} + s_\perp \xi_{m,j}^\perp + \sum_{k=3}^d s_ke_k:\ & |s_\parallel| \le \frac{c\eta}{2}\delta_m^{-2}, |s_\perp| \le \frac{c\eta}{2}\delta_m^{-1}, \qtq{and} \\
 & \abs*{s_k} \le \frac{c\eta}{2}\delta_m^{-1}, \ 3 \le k \le d\}.
 \end{aligned}
\end{equation}
Using \eqref{eq: compresssphere}, we get
\begin{equation}\label{eq: sphcomp}
    |\bigcup_{j = 1}^{2^m}Q_{m,j}| \le \rho_mS_m, \qquad \rho_m: = C_0 \frac{\log m}{m}.
\end{equation}
For $1 \le j \le 2^m$, define the spherical cap
\[
\Theta_{m,j}: = \{ \xi \in \Sk: \abs*{\xi - \xi_{m,j}} \le \delta_m\}.
\]
In particular, $\xi_{m,j} \in \Theta_{m,j}$. Set
\[
\Omega_m := \bigcup_{j=1}^{2^m} \Theta_{m,j}, \qquad a_m := \sigma (\Omega_m).
\]
By \eqref{eq: dirsep}, whenever $j \neq k$, $\abs*{\xi_{m,j} - \xi_{m,k}} \ge c_{\mathrm{sep}} \delta_m$. Hence, the caps $\Theta_{m,j}$ have bounded overlap: there exists a constant $M_d > 0$, independent of $m$, such that
\[
\sum_{j = 1}^{2^m} \rchi_{\Theta_{m,j}} \le M_d, \qquad \xi \in \Sk.
\]
We may, for example, take
\[
M_d := 1 + \big(1 + \frac{2}{c_{\mathrm{sep}}}\big)^d.
\]
There exists a constant $c_{\mathrm{cap}} > 0 $, depending only on $d$, such that $\sigma(\Theta_{m,j}) \ge c_\mathrm{cap}\delta_m^{d-1}$, and consequently,
\begin{equation}\label{eq: boverlapcon}
    a_m = \sigma(\Omega_m)
 \ge c_\Omega2^m\delta_m^{d-1}, \qquad  c_\Omega := \frac{c_{\mathrm{cap}}}{M_d}.
\end{equation}
 For $\xi \in \Theta_{m,j}$, write 
 \[
 \xi-\xi_{m,j} = t_\parallel \xi_{m,j} + t_\perp \xi_{m,j}^\perp + \sum_{k=3}^d t_ke_k.
 \]
It follows that
\begin{equation}\label{eq: capbox}
     |t_\parallel| \le \frac{1}{2}\delta_m^{2}, \qquad
 |t_\perp| \le \delta_m,  \qquad
 |t_k| \le \delta_m, \qquad \ 3 \le k \le d.
\end{equation}
For $x \in Q_{m,j}$, using \eqref{eq: Qmj}, we can write 
\[
 x-\lambda_{m,j} = s_\parallel \xi_{m,j} + s_\perp \xi_{m,j}^\perp + \sum_{k=3}^d s_ke_k,
\]
where
\begin{equation}\label{eq: sphfre}
  |s_\parallel| \le \frac{c\eta}{2}\delta_m^{-2},\qquad
 |s_\perp| \le \frac{c\eta}{2}\delta_m^{-1},\qquad
 \abs*{s_k} \le \frac{c\eta}{2}\delta_m^{-1}, \qquad 3 \le k \le d.  
\end{equation}
Using the orthonormality of the vectors\ $\xi_{m,j}$, $\xi_{m,j}^\perp$, $e_3$, $\dots$, $e_d$\ in $\Rd$ together with \eqref{eq: capbox} and \eqref{eq: sphfre}, we obtain
\[
\abs*{\left( x-\lambda_{m,j}\right) \cdot\left( \xi-\xi_{m,j}\right) } \le \frac{\left(2d-1\right)c\eta}{4}.
\]
We write $x = (x-\lambda_{m,j}) + \lambda_{m,j}$, use $\eta = 1/6cd$ and $\cos{(2 \pi y) \ge 1/2}$ for $|y| \le 1/6$ to get
\begin{equation}\label{eq: phasesphere}
     |\int_{Q_{m,j}} e^{-2 \pi i x \cdot \left( \xi - \xi_{m,j}\right)}\,\dd x| \ge \frac{1}{2} \abs*{Q_{m,j}}.
\end{equation}
Every $Q_{m,j}$ has the same volume 
$q_m := \abs*{Q_{m,j}} = (c \eta)^d\delta_m^{-(d+1)}$, and 
\[
S_m : = \sum_{j=1}^{2^m}\abs*{Q_{m,j}} = 2^m q_m.
\]
By \eqref{eq: boverlapcon},
\begin{equation}\label{eq: scalesphere}
    \frac{q_ma_m^{1/p_d}}{S_m^{1/p_d}} \ge c_\Omega^{1/p_d}q_m^{1-1/p_d}\delta_m^{(d-1)/p_d} = c_\Omega^{1/p_d}(c\eta)^{(d-1)/2}.
\end{equation}

\begin{proof}[Proof of \Cref{thm: sphere}]
     For sufficiently large $m$, define the ordered datum
\[
\scriptP_m^\mathbb{S} := \left\{\left(Q_{m,j}, \xi_{m,j}, \Theta_{m,j}\right)\right\}_{j = 1}^{2^m},
\]
where the order is inherited from the ordering of the triangles $T_{m,j}$.

Using the selection rule provided by \Cref{thm: abstractthm}, set $E_m^{\mathbb{S}} : = E(\scriptP_m^{\mathbb{S}})$. This prescribes the sequence  $(E_m^{\mathbb{S}})$. Using \eqref{eq: sphcomp}, \eqref{eq: phasesphere}, and \eqref{eq: scalesphere}, we obtain 
\[
 \rho(\scriptP_m^\mathbb{S}) \le C_0 \frac{\log m}{m}, \qquad \kappa(\scriptP_m^\mathbb{S}) \ge \frac{1}{2}, \qquad b_{p_d}(\scriptP_m^\mathbb{S}) \ge b_d,
\]
where $C_0 >0$ and $b_d := c_\Omega^{1/p_d}(c\eta)^{(d-1)/2} >0$ are independent of $m$. Since $\theta_{p_d} = 1/(2d)$, \Cref{thm: abstractthm} gives 
\[
\begin{aligned}
     \frac{\| \hat{\rchi_{E_m^{\mathbb{S}}}}\|_{L^{p_d, \infty}(\Sk, \sigma)}}{\abs*{E_m^{\mathbb{S}}}^{1/p_d}} &\ge c_{p_d}\kappa\left(\scriptP_m^\mathbb{S}\right) b_{p_d} \left(\scriptP_m^\mathbb{S}\right)\rho\left(\scriptP_m^\mathbb{S}\right)^{-1/(2d)} \ge C_\mathbb{S}\left(\frac{m}{\log m}\right)^{1/(2d)},
\end{aligned} 
\]
where $ C_\mathbb{S} :=  (c_{p_d}b_d/2)C_0^{-1/(2d)}$. Letting $m \rightarrow \infty$ finishes the proof of \Cref{thm: sphere}.
\end{proof}

 \subsection{Paraboloid}\label{sec: secpar}

To prove \Cref{thm: parabola}, we start with the family of triangles given in \cite{keich}. We choose a parallelogram inside each triangle, apply the same linear map to all these parallelograms, and take their Cartesian products with the same box. We then apply another linear map to all of the parallelepipeds and use the resulting family to form an ordered datum for the paraboloid.

 We first give the construction for $d \ge 3$. Choose a closed interval $J = [u,v]$, a point $\zeta'\in \R^{d-2}$, and $r_I>0$ such that $[u-r_I, v+r_I] \times(\zeta'+ [-r_I, r_I]^{d-2}) \subset \operatorname{int}I$. Let $\ell := v-u >0$. For every integer $m \ge 1$, set $\delta_m : = 2^{-m}$, $N_m := 2^m = \delta_m^{-1}$.

 For $0 \le k < N_m$, let $\varepsilon_i(k) \in \{0,1\}, 1\le i \le m$ be the unique binary digits such that $k = \sum_{i=1}^m \varepsilon_i(k) 2^{m-i}$, 
 and set
 \[
 s_{m,k} := k \delta_m = \sum_{i=1}^m \varepsilon_i(k)2^{-i}, \qquad b_{m,k} :=  \sum_{i=1}^m \frac{1-i}{m}  \varepsilon_i(k)2^{-i}.
 \]
 Define 
 \[
 p_{m,k}:= (0,b_{m,k}), \qquad v_{m,k} := \left(1,s_{m,k}\right), \qquad w_m:= \left(0,\ -\delta_m\right),
 \]
 and
 \[
 T_{m,k} :=  p_{m,k} + \left\{ \alpha v_{m,k} + \beta  w_m:  \alpha,\beta  \ge 0, \ \alpha+\beta \le 1 \right\}.
 \]
Equivalently, $ T_{m,k}$ is the triangle with vertices 
\[
\left(0,b_{m,k}\right), \qquad \left(0,b_{m,k} - \delta_m\right), \qquad \left(1 , b_{m,k} + s_{m,k}\right).
\]
The proof of \cite[Theorem 1]{keich}, with $n = m$, gives
\begin{equation}\label{eq: keichcomp}
    |\bigcup_{k = 0}^{N_m-1} T_{m,k}| < \frac{1}{m}.
\end{equation}
Inside $ T_{m,k}$, define the parallelogram
\[
 P_{m,k} :=  p_{m,k} + \{ \alpha v_{m,k} + \beta  w_m: \frac{1}{4} \le \alpha \le \frac{1}{2}, \ \frac{1}{4} \le \beta \le \frac{1}{2} \}.
\]
The containment $P_{m,k} \subset  T_{m,k}$ follows from $\alpha + \beta \le 1$. Moreover,
\[
|P_{m,k}| = \frac{\delta_m}{16}, \qquad  \sum_{k = 0}^{N_m-1}\abs*{P_{m,k}} = \frac{1}{16}.
\]
It follows from \eqref{eq: keichcomp} that
\begin{equation}\label{eq: keichcomppar}
    \frac{|\bigcup_{k = 0}^{N_m-1} P_{m,k}|}{ \sum_{k = 0}^{N_m-1}\abs*{P_{m,k}} } < \frac{16}{m}.
\end{equation}
Define $t_{m,k} := u+ \ell s_{m,k} = u + \ell k \delta_m$ for $0 \le k < N_m$. Note that $|t_{m,k}-t_{m,l}| = \ell \delta_m|k-l|$, and consider the invertible linear map $A_J(x_1,x_2) := (-2ux_1 - 2 \ell x_2, x_1)$. This map is independent of $k$, has determinant $2 \ell$, and 
\begin{equation}\label{eq: parquan}
    A_Jv_{m,k} = (-2t_{m,k}, 1), \qquad  A_Jw_m = 2 \ell \delta_m e_1.
\end{equation}
Define 
\[
 R_{m,k} : =  \delta_m^{-2}A_JP_{m,k}.
\]
Writing $\alpha = (3/8) +a,\ \beta = (3/8) +b$ for $\abs*{a}, \abs*{b} \le 1/8$, and using \eqref{eq: parquan}, we see that $R_{m,k}$ is a translation of
\[
R_{m,k,c} := \{ r \left( -2t_{m,k},1\right) + s_1e_1 : |r| \le \frac{1}{8}\delta_m^{-2}, \ |s_1| \le \frac{\ell}{4} \delta_m^{-1}\}.
\]
Since the same invertible linear map and dilation are applied to all the  parallelograms $P_{m,k}$, \eqref{eq: keichcomppar} gives
\begin{equation*}
    \text{Comp}(\{ R_{m,k}\}_{k=0}^{N_m-1}) < \frac{16}{m}.
\end{equation*}
For
\[
B_m:= \{\left(s_2, \dots, s_{d-1}\right) \in \R^{d-2}: \abs*{s_j} \le \frac{1}{4} \delta_m^{-1},\ 2 \le j \le d-1\},
\]
define $\widetilde Q_{m,k} : = \{(x_1, s_2, \dots, s_{d-1},x_d) \in \Rd: (x_1,x_d) \in R_{m,k},\ (s_2, \dots, s_{d-1})\in B_m\}$. Equivalently,  $\widetilde Q_{m,k} = \Phi(R_{m,k} \times B_m)$, where $\Phi: \R^2 \times \R^{d-2} \rightarrow \Rd$ is defined by $\Phi((a,b), (s_2, \dots, s_{d-1})) := (a,s_2, \dots, s_{d-1},b)$. Consequently, 
\[
\text{Comp}(\{ \widetilde Q_{m,k}\}_{k=0}^{N_m-1}) = \text{Comp}(\{ R_{m,k}\}_{k=0}^{N_m-1}).
\]
Define the linear map 
\[
L_{\zeta'}(x_1, x',x_d):=(x_1,x'-2\zeta'x_d,x_d),
\]
where $x':= (x_2, \dots, x_{d-1}) \in \R^{d-2}$. This map has determinant $1$, fixes $e_1, \dots ,e_{d-1}$, and satisfies $L_{\zeta'}(-2t_{m,k},0, \dots, 0, 1) = (-2t_{m,k}, -2\zeta',1)$. 

Let $Q_{m,k} := L_{\zeta'}(\widetilde Q_{m,k})$. Since the same map is applied to every $\widetilde Q_{m,k}$, we obtain
\begin{equation}\label{eq: geomcompar}
    \text{Comp}(\scriptQ_m^{\Gamma_{d}})< \frac{16}{m}, \qquad \scriptQ_m^{\Gamma_{d}}:= \{ Q_{m,k}\}_{k=0}^{N_m-1}.
\end{equation}
Let $\zeta_{m,k} := (t_{m,k}, \zeta') \in \R^{d-1}$. Every $Q_{m,k}$ is a translation of 
\[
\begin{aligned}
Q_{m,k,c}:= \{r\left(-2\zeta_{m,k},1\right)+\sum_{j=1}^{d-1} s_je_j:|r| \le \frac{1}{8} \delta_m^{-2},\ &|s_1| \le \frac{\ell}{4} \delta_m^{-1},\\
&|s_j| \le \frac{1}{4} \delta_m^{-1},\ 2\le j \le d-1\}.
\end{aligned}
\]
Every $Q_{m,k}$ has the same volume $q_m := |Q_{m,k}| = (\ell/2^{(d+1)})\delta_m^{-(d+1)}$.

Choose a positive $c_\Theta$ such that 
\begin{equation}\label{eq: cTheta}
 2c_\Theta < \ell, \qquad \frac{(d-1)c_\Theta^2}{8} +  \frac{(\ell +d-2) c_\Theta}{4} \le \frac{1}{6}.
\end{equation}
For all sufficiently large $m$, define
\[
I_{m,k}:= \left\{\zeta_{m,k}+\tau: |\tau_j| \le c_\Theta\delta_m,\ 1 \le j \le d-1\right\},\qquad 0\le k < N_m.
\]
These sets are pairwise disjoint and contained in $I$. Indeed, disjointness follows from $|t_{m,k}-t_{m,l}| = \ell \delta_m|k-l|$ and $ 2c_\Theta < \ell$, while containment follows from the choice of $J,\zeta'$, and $r_I$ once $m$ is sufficiently large. Set
\[
\Theta_{m,k} := \Gamma_d\left(I_{m,k}\right), \quad \xi_{m,k}:= \Gamma_d\left(\zeta_{m,k}\right), \quad \Omega_m:= \bigcup_{k=0}^{N_m-1} \Theta_{m,k}, \quad a_m := \mu_{\Gamma_d}\left(\Omega_m\right).
\]
Since the sets $I_{m,k}$ are pairwise disjoint, we have $a_m = (2c_\Theta)^{d-1}\delta_m^{d-2}$.

Since $t_{m,0} < \dots < t_{m,k} < \dots <t_{m, N_m-1}$, the index $k$ provides a natural ordering. For sufficiently large $m$, define the ordered datum
\[
\scriptP_m^{\Gamma_d} := \left\{\left(Q_{m,k}, \xi_{m,k}, \Theta_{m,k}\right)\right\}_{k = 0}^{N^m-1}.
\]
Let $\eta \in I_{m,k}$. Then $\eta = \zeta_{m,k} + \tau$, where $|\tau_j| \le c_{\Theta}\delta_m$ for $1 \le j \le d-1$. Write a point $x_c$ of $Q_{m,k,c}$ as 
\[
x_c = r\left(-2\zeta_{m,k},1\right)+\sum_{j=1}^{d-1} s_je_j.
\]
Since
\[
\Gamma_d(\eta) - \Gamma_d(\zeta_{m,k}) = (\tau, 2\zeta_{m,k} \cdot \tau + |\tau|^2),
\]
we have, by \eqref{eq: cTheta},
\begin{equation*}
    |x_c \cdot \left(\Gamma_d(\eta) - \Gamma_d(\zeta_{m,k})\right)| = |r|\tau|^2 + \sum_{j=1}^{d-1}s_j\tau_j| \le \frac{(d-1)c_\Theta^2}{8} +  \frac{(\ell+d-2) c_\Theta}{4} \le \frac{1}{6}.
\end{equation*}
This estimate together with $\cos{(2 \pi y) \ge 1/2}$ for $|y| \le 1/6$, gives
\begin{equation}\label{eq: phacopar}
  |\int_{Q_{m,k}} e^{-2 \pi i x \cdot \left(\Gamma_d(\eta) - \Gamma_d(\zeta_{m,k})\right)}\, \dd x| \ge \frac{1}{2} \abs*{Q_{m,k}}.  
\end{equation}
Put 
\[
S_m : = \sum_{k = 0}^{N_m-1}\abs*{Q_{m,k}} = N_mq_m.
\]
Since  $q_m = (\ell/2^{d+1})\delta_m^{-(d+1)},\ a_m = (2c_\Theta)^{d-1}\delta_m^{d-2}$, and $N_m = \delta_m^{-1}$, it follows that
\begin{equation}\label{eq: cribapar}
  q_ma_m^{1/p_d} = \beta_{d,I}S_m^{1/p_d}, \qquad \beta_{d,I} := \left(\frac{\ell}{2^{d+1}}\right)^{1-(1/p_d)} \left(2c_\Theta\right)^{(d-1)/p_d} > 0.  
\end{equation}
For $d=2$, choose a closed interval $J= [u,v]$ and $r_I>0$ such that $[u-r_I, v+r_I]\subset \operatorname{int}I$. Set $Q_{m,k} := R_{m,k}$ and $\zeta_{m,k} := t_{m,k}$. With these choices,  \eqref{eq: geomcompar}, \eqref{eq: phacopar}, and \eqref{eq: cribapar} also hold.

\begin{proof}[Proof of \Cref{thm: parabola}]
For each sufficiently large $m$, let $E_m^{\Gamma_d} := E(\scriptP_m^{\Gamma_d})$ be the set selected by \Cref{thm: abstractthm}. This prescribes the sequence $( E_m^{\Gamma_d})$. Using \eqref{eq: geomcompar}, \eqref{eq: phacopar}, and \eqref{eq: cribapar}, we obtain
\[
\rho(\scriptP_m^{\Gamma_d}) < \frac{16}{m}, \qquad \kappa(\scriptP_m^{\Gamma_d}) \ge \frac{1}{2}, \qquad b_{p_d}(\scriptP_m^{\Gamma_d}) = \beta_{d,I},
\]
where $\beta_{d,I}$ is independent of $m$. Since $\theta_{p_d} = 1/(2d)$, \Cref{thm: abstractthm} gives
\[
 \frac{\| \hat{\rchi_{E_m^{\Gamma_d}}}\circ \Gamma_d\|_{L^{p_d, \infty}(I,d\zeta)}}{|E_m^{\Gamma_d}|^{1/p_d}} \ge c_{p_d}\frac{1}{2}\beta_{d,I} \left(\frac{16}{m}\right)^{-1/(2d)}.
\]
Therefore, 
\[
\frac{\| \hat{\rchi_{E_m^{\Gamma_d}}}\circ \Gamma_d\|_{L^{p_d, \infty}(I,d\zeta)}}{|E_m^{\Gamma_d}|^{1/p_d}} \ge C_{\Gamma_d,I}m^{1/(2d)},
\]
where
\[
C_{\Gamma_d,I}:= c_{p_d}\frac{1}{2}\beta_{d,I}16^{-1/(2d)} >0
\]
depends on $d$ and $I$, but is independent of $m$. Letting $m \longrightarrow \infty$ proves \Cref{thm: parabola}.
\end{proof}

\section{Moment-curve obstruction} \label{sec : moment curve}

In this section, we prove \Cref{thm: momemntobstruct} and \Cref{cor: momcor}. For $1 \le r \le d$ and $t \in \R$, set
\[
X_r^t(x) := \gamma_d^{(r)}(t) \cdot x , \qquad x \in \Rd.
\]
We note that since $\gamma_d^{(k)}(t) \cdot u_r(t) = \delta_{kr}$ for $1 \le r, k \le d$, a point $x$ belongs to $Q_{\delta,t}(c)$ if and only if 
\begin{equation}\label{eq: intervals}
    |X_r^t(x)-X_r^t(c)| \le \delta^{-r}, \qquad 1 \le r \le d.
\end{equation}
Thus $X_r^t$ ranges over an interval of length $2\delta^{-r}$ on $Q_{\delta,t}(c)$. Moreover, $|Q_{\delta,t}(c)| \sim \delta^{-D}$.

\begin{lemma}\label{lem: dercoord}
    Let $s, t \in R$, put $h := s-t$, and let $1 \le r \le d$. Then 
\begin{equation}\label{eq: dertaylor}
    X_r^s(x) = \sum_{j = 0}^{d-r} \frac{h_j}{j!}X_{r+j}^t(x), \qquad x \in \Rd.
\end{equation}
\end{lemma}
\begin{proof}
    Since $\gamma_d$ is a polynomial of degree $d$, Taylor's formula gives:
    \[
     \gamma_d^{(r)}(s) = \sum_{j = 0}^{d-r} \frac{h_j}{j!}\gamma_d^{(r+j)}(t).
    \]
    Taking the dot product with $x$ proves the claim.
\end{proof}

Applying the above formulas to the cases $r=d-2$ and $r=d-1$, respectively, yields
\begin{align}
    X_{d-2}^s &= X_{d-2}^t + hX_{d-1}^t + \frac{h^2}{2}X_{d}^t,
\label{eq: dmin2tayl}\\
     X_{d-1}^s &= X_{d-1}^t + hX_{d}^t.\label{eq: dmin1tayl}
\end{align}
Therefore, using \eqref{eq: dmin1tayl} to write $hX_{d}^t = X_{d-1}^s - X_{d-1}^t$ and inserting this into \eqref{eq: dmin2tayl}, we obtain
\begin{equation}\label{eq: dmin12tayl}
    X_{d-2}^s-X_{d-2}^t = \frac{h}{2}(X_{d-1}^t+X_{d-1}^s).
\end{equation}

\begin{proof}[Proof of \Cref{thm: momemntobstruct}(i)]

Fix $s, t \in \R$ with $s \ne t$, choose arbitrary $c_s, c_t \in \Rd$, and write
\[
Q_t := Q_{\delta,t}(c_t), \qquad Q_s := Q_{\delta,s}(c_s), \qquad h:= s-t.
\]
Consider the coordinate map
\[
\Phi_t : \Rd \longrightarrow \Rd, \qquad \Phi_t(x) := (X_1^t(x), \dots, X_d^t(x)).
\]
Let $E := \Phi_t(Q_t \cap Q_s) $. It follows that 
\begin{equation}\label{eq: comap}
    |E| \sim |Q_t \cap Q_s|.
\end{equation}
For $x \in Q_t \cap Q_s$, write a point of $E$ as $y = (y_1, \dots, y_d) = (X_1^t(x), \dots, X_d^t(x))$, and let $z_r := X_r^s(x), 1 \le r \le d$. By \eqref{eq: intervals}, each of $y_r$ and $z_r$ can lie in an interval of length at most $2\delta^{-r}$, although the centers of these intervals may depend on $c_t$ and $c_s$.

Fix a value of $y_{d-1}$. By \eqref{eq: dmin12tayl}, 
\[
z_{d-2} - y_{d-2} = \frac{h}{2}(y_{d-1} + z_{d-1}).
\]
It follows that, for this fixed value of $y_{d-1}$, the coordinate $z_{d-1}$ lies in an interval of length at most $\lesssim \frac{\delta^{-(d-2)}}{|h|}$.
On the other hand, \eqref{eq: dmin1tayl} gives $z_{d-1} = y_{d-1} + hy_d$. Therefore, for the same fixed value of $y_{d-1}$, the coordinate $y_d$ lies in an interval of length at most $\lesssim \frac{\delta^{-(d-2)}}{|h|^2}$. Thus, for $E^* := \{(y_{d-1}, y_d): (y_1, \dots, y_d) \in E\}$, by Fubini,
\[
\abs*{E^*} \lesssim \delta^{-(d-1)}\frac{\delta^{-(d-2)}}{|h|^2}.
\]
In consequence,
\[
|E| \lesssim (\prod_{r = 1}^{d-2}\delta^{-r})\abs*{E^*} \lesssim \frac{\delta^{-(D-2)}}{|h|^2}.
\]
Combining this with \eqref{eq: comap} proves part (i).
\end{proof}

\begin{proof}[Proof of \Cref{thm: momemntobstruct}(ii)]
    
If $T_\delta$ is empty, the assertion is immediate. Otherwise, for $N \ge 1$, write
\[
T_{\delta} = \left\{t_1 < \dots <t_N\right\}, \qquad Q_j := Q_{\delta, {t_j}}(c_{t_j}), \qtq{and} U := \bigcup_{j = 1}^NQ_j.
\]
Since $T_\delta$ is $\delta$-separated, $    \abs*{t_{j\mathbin{\pm}m}-t_j}\ge m\delta,\ m\ge1$, whenever $j\mathbin{\pm}m$ lies between $1$ and $N$. 

By part (i) and using the fact that $|Q_{\delta,t}(c)| \sim \delta^{-D}$ we obtain
\begin{equation*}\label{eq: quadcay}
    \abs*{Q_j \cap Q_{j \mathbin{\pm}m}} \lesssim \frac{\delta^{-(D-2)}}{(m\delta)^2} \sim \frac{\abs*{Q_j}}{m^2}.
\end{equation*}
Thus
$
    \sum_{\substack{1 \le k \le N \\k \ne j}}\abs*{Q_j \cap Q_k} \lesssim \abs*{Q_j}.
$
Consequently, the function $F(x) := \sum_{j = 1}^N\rchi_{Q_j}(x)$ satisfies
\begin{equation*}\label{eq: squmul}
    \int_{\Rd}F^2 \lesssim \sum_{j = 1}^N \abs*{Q_j}.
\end{equation*}
Note that $\supp(F)$ is contained in $U$, and so applying Cauchy-Schwarz we obtain
\[
( \sum_{j = 1}^N \abs*{Q_j})^2 = (\int_UF)^2 \le \abs*{U}\int_{\Rd}F^2 \leq  |U| \sum_{j = 1}^N \abs*{Q_j},
\]
 completing the proof of part (ii).
\end{proof}

 To summarize, in dimension $d\ge3$, whenever the points $t_1, \dots, t_N \in \R$ are $\delta$-separated, we have 
 \[
  \text{Comp}(\{Q_j\}_{j=1}^N) \sim1,
 \]
 uniformly over all translations. This is a fundamental difference between dimension two and higher. In dimension two, if $N\ge2$ and the points $t_1, \dots, t_N$ are $\delta$-separated, we have 
 \begin{equation}\label{eq: maindif}
      \text{Comp}(\{Q_j\}_{j=1}^N) \gtrsim \frac{1}{\log N}.
 \end{equation}
 Indeed, the corresponding intersection estimate gives an upper bound of order $|Q_j|/r$ when the indices differ by $r$. The argument used in the proof of \Cref{thm: momemntobstruct}(ii) then gives \eqref{eq: maindif}. This does not guarantee that every family has a compression ratio comparable to $1/(\log N)$; some may have a larger ratio. However, the family constructed in \Cref{sec: secpar} based on Keich's construction attains this order. Indeed, using \eqref{eq: maindif}, \eqref{eq: geomcompar}, and $N_m = 2^m$, we obtain
 \[
    \text{Comp}(\scriptQ_m^{\gamma_2}) \sim \frac{1}{\log N_m}.
 \]
 \begin{proof}[Proof of \Cref{cor: momcor}]
     
  Let $\scriptP = \{(Q_{\delta, t_j}(c_j), \gamma_d(t_j), \gamma_d(I_j)\}_{j = 1}^N$ be a single-scale datum as defined in \Cref{sec: momentobs}, and set
\[
q : = |Q_{\delta, t_j}(c_j)|, \qquad a := \mu_{\gamma_d}(\bigcup_{j=1}^N\gamma_d(I_j)) = |\bigcup_{j=1}^NI_j|.
\]
We have $q \sim \delta^{-D}$. For $r_d = 1 + 1/D$,
\begin{equation}\label{eq: compat}
    b_{r_d}\left(\scriptP\right)^{r_d} \sim \frac{|\bigcup_{j=1}^NI_j|}{N\delta}. 
\end{equation}
Choose a maximal $\delta$-separated subcollection of $\{t_j\}_{j = 1}^N$, and let $J$ be its index set and $M := |J|$. Maximality implies that for every $1 \le i \le N$, there exists some $j \in J$ such that $|t_i-t_j| < \delta$. Since $t_i \in I_i$ and $|I_i| \le A\delta$,
\[
\bigcup_{i = 1}^NI_i \subset \bigcup_{j \in J}\left[t_j-\left(A+1\right)\delta, t_j+\left(A+1\right)\delta\right].
\]
It follows from \eqref{eq: compat} that
\begin{equation}\label{eq: fun1}
      b_{r_d}\left(\scriptP\right)^{r_d} \lesssim_A \frac{M}{N}.
\end{equation}
Applying \Cref{thm: momemntobstruct}(ii) to the parallelepipeds corresponding to the indices in $J$, we get
\[
|\bigcup_{j = 1}^N Q_{\delta, t_j}(c_j)| \ge |\bigcup_{j \in J} Q_{\delta, t_j}(c_j)| \gtrsim Mq.
\]
Thus, $\rho\left(\scriptP\right) \gtrsim \frac{M}{N}$. Combining this with \eqref{eq: fun1} gives $b_{r_d}\left(\scriptP\right)^{r_d} \lesssim_A  \rho\left(\scriptP\right)$. Since $\kappa(\scriptP) \le 1$(recalling that $\theta_{r_d} = 1/r_d - 1/2$), we deduce
$$
\kappa(\scriptP) b_{r_d}(\scriptP) \rho(\scriptP)^{- \theta_{r_d}}  \lesssim_A \rho(\scriptP)^{1/r_d-\theta_{r_d}} = \rho\left(\scriptP\right)^{1/2} \lesssim_A 1.
 $$
In other words, for any sequence $(\scriptP_m)_{m\ge 1}$ of such data, the above holds uniformly in $m$, so the LHS cannot diverge. Moreover,
\[
 \kappa(\scriptP_m) b_{r_d}(\scriptP_m) \rho(\scriptP_m)^{-\theta_{r_d}} \longrightarrow 0 \qtq{as} \rho(\scriptP_m) \rightarrow0.
\]
This completes the proof.
\end{proof}

\section{Sharpness of the exponent \texorpdfstring{$-\theta_p$}{minus theta p} of \texorpdfstring{$\rho(\scriptP)$}{rho(P)}}\label{sec : sec5}

\begin{proof}[Proof of \Cref{prop: sharpness}]
For a fixed $N \ge 2$ and each $1 \le j \le N$, we set
\[
Q_{N,j} := [0,1]^d, \qquad \xi_{N,j} := 2^je_1, \qtq{and} \Theta_{N,j} := \left\{\xi_{N,j}\right\}
\]
where $e_1 = (1, 0, \ldots, 0)$ denotes the first basis vector. We equip the set $\Sigma_N:= \{\xi_{N,1}, \dots, \xi_{N,N}\}$ with counting measure $\mu_N$. Let us order the triples by increasing $j$, and consider $\scriptP_N = \left\{\left(Q_{N,j}, \xi_{N,j}, \Theta_{N,j}\right)\right\}_{j = 1}^N$. Thus, $U_N = [0, 1]^d,\ q = 1,\ S_N:=  \sum_{j = 1}^N \abs*{Q_{N,j}} = N$, and $a_N := \mu_N(\bigcup_{j=1}^N \Theta_{N,j}) = N$, and so
\[
\rho\left(\scriptP_N\right) = \frac{\abs*{U_N}}{S_N} = \frac{1}{N}, \qquad  b_p\left(\scriptP_N\right) = \frac{qa_N^{1/p}}{S_N^{1/p}} = 1.
\]
Moreover, we note that for the single point $\xi = \xi_{N,j} \in \Theta_{N,j}$,
\[
|\int_{Q_{N,j}} e^{-2 \pi i x \cdot (\xi - \xi_{N,j})}\,\dd x| = q,
\]
resulting in $\kappa(\scriptP_N) = 1$.

We now consider the upper bound. We first apply the standard lacunary-series inequality (see \cite[Theorem 3.6.4]{Grafakos}) and then H\"older's (recalling that $p' > 2$), to get
\[
\|\sum_{j=1}^N v_j e^{2 \pi i 2^j x_1}\|_{L^{p^\prime}([0, 1]^d)} \lesssim \left\|v\right\|_{\ell^2} \le N^{\theta_p}\left\|v\right\|_{\ell^{p^\prime}}
\]
for each $v = (v_1, \dots, v_N)$. As a result for the following vector valued operator
\[
\scriptR_Nf := \big(\int_{[0, 1]^d}f(x)e^{-2 \pi i 2^j x_1}\, \dd x\big)_{j = 1}^N
\]
we obtain the estimate $\left\|\scriptR_Nf\right\|_{\ell^p} \lesssim N^{\theta_p}\left\|f\right\|_{L^p([0, 1]^d)}$. This yields
\[
\|\hat{\rchi_E}\|_{L^{p,\infty}\left(\Sigma_N, \mu_N\right)} \le \left\|\scriptR_N\rchi_E\right\|_{\ell^p} \lesssim N^{\theta_p}\abs*{E}^{1/p}
\]
for each measurable $E \subset [0, 1]^d$. In other words,
\[
\sup_{\substack{E \subset U_N\\0 < \abs*{E} < \infty}} \frac{\|\hat{\rchi_E}\|_{L^{p,\infty}\left(\Sigma_N, \mu_N\right)}}{\abs*{E}^{1/p}} \lesssim N^{\theta_p} \sim \rho\left(\scriptP_N\right)^{-\theta_p}.
\]
This completes the proof.
\end{proof}

\section{Acknowledgments.}
The author is grateful to his PhD adviser, Chandan Biswas, for suggesting the problem studied in this work and for his invaluable guidance throughout the project. His many thoughtful suggestions contributed significantly to the development of this manuscript.



\bibliographystyle{plain}
\bibliography{fourierP1}


\end{document}